\documentclass[11pt,a4paper,fleqn]{article}

\usepackage{amsthm}
\usepackage{amssymb}
\usepackage{amstext}
\usepackage{amsmath}
\usepackage[dvips]{	graphicx}
\usepackage{url}
\usepackage{hyperref} \hypersetup{backref, colorlinks=true, linkcolor=magenta,citecolor=blue}
\usepackage[margin = 40pt,font=small,labelfont=bf]{caption}
\usepackage{arydshln}
\usepackage[latin1]{inputenc}
\usepackage[T1]{fontenc}
\usepackage{color}

\numberwithin{equation}{section}

\newcommand{\thefont}[2]{\fontsize{#1}{#2}\fontshape{n}\selectfont}
\newcommand{\ind}{\rlap{\thefont{10pt}{12pt}1}\kern.16em\rlap{\thefont{11pt}{13.2pt}1}\kern.4em}
\def\argmin{\mathop{\rm arg \; min}\limits}%

\newtheorem{prop}{Proposition}[section]
\newtheorem{coro}{Corollary}[section]
\newtheorem{theo}{Theorem}[section]

\newtheorem{lem}{Lemma}[section]
\newtheorem{ass}{Assumption}
\newtheorem{rem}{Remark}[section]

\renewcommand{\P}{{\mathbb P}}

\newcommand{\E}{{\mathbb E}}
\newcommand{\EE}{{\mathbb E}}

\newcommand{\N}{{\mathbb N}}

\newcommand{\RR}{{\mathbb R}}

\newcommand{\FF}{{\cal F}}

\newcommand{\Ss}{{\mathbb S}}

\newcommand{\BB}{\ensuremath{\mathcal B}}

\newcommand{\MM}{\ensuremath{\mathcal M}}

\newcommand{\PP}{\ensuremath{\mathcal P}}

\newcommand{\1}{\rlap{\thefont{10pt}{12pt}1}\kern.16em\rlap{\thefont{11pt}{13.2pt}1}\kern.4em}

\newcommand{\btheta}{\mbox{\boldmath$\theta $}}

\newcommand{\bvarphi}{\mbox{\boldmath$\varphi $}}
\newcommand{\bb}{{\mathbf b}}
\newcommand{\bq}{{\mathbf q}}

\newcommand{\ba}{{\mathbf a}}
\newcommand{\bY}{{\mathbf Y}}
\newcommand{\bX}{{\mathbf X}}

\newcommand{\bmu}{\mbox{$\boldsymbol{\mu}$}}

\title{Characterization of barycenters in the Wasserstein space by averaging optimal transport maps}

\author{ J\'er\'emie Bigot$^{1}$\footnote{J. Bigot is a member of Institut Universitaire de France.}  \hspace{0.05cm}  and  Thierry Klein$^{2,3}$  \vspace{0.3cm}  \  \\
 Institut de Math\'ematiques de Bordeaux et CNRS  (UMR 5251)$^{1}$  \\ Universit\'e de Bordeaux  \vspace{0.1cm}  \\ 
ENAC - Ecole Nationale de l'Aviation Civile \\ Universit\'e de Toulouse, France
\vspace{0.1cm}  \\ 
Institut de Math\'ematiques de Toulouse et CNRS  (UMR 5219)$^{3}$ \\ Universit\'e de Toulouse  }
 
\date{\today}

\begin{document}

\maketitle

\thispagestyle{empty}

\begin{abstract}
This paper is concerned by the study of barycenters for random probability measures  in the Wasserstein space. Using a duality argument, we give a precise characterization of the population barycenter for various parametric classes of random probability measures with compact support. In particular, we make a connection between averaging in the  Wasserstein space as introduced in Agueh and Carlier \cite{MR2801182}, and taking the expectation of optimal transport maps with respect to a fixed reference measure. We also discuss the usefulness of this approach in statistics for the analysis of deformable models in signal and image processing. In this setting, the problem of estimating a population barycenter from $n$ independent and  identically distributed random probability  measures is also considered. 
\end{abstract}



\hspace{\baselineskip}

\noindent\textbf{Acknowledgements}

\noindent  The authors acknowledge the support of the French Agence Nationale de la Recherche (ANR) under reference ANR-JCJC-SIMI1 DEMOS. We would like to also thank J\'er\^ome Bertrand for fruitful discussions on Wasserstein spaces and the optimal transport problem. We would like to thank the referees and the associate editor for their comments and suggestions. 
\section{Introduction}

In this paper, we consider the problem of characterizing the barycenter of random probability measures on $\RR^{d}$. The set of  Radon probability measures endowed with the 2-Wasserstein distance is not an Euclidean space. Consequently, to define a notion of barycenter for  random probability measures, it is natural to use the notion of Fr\'echet mean \cite{fre} that is an extension of the usual Euclidean barycenter to non-linear spaces endowed with non-Euclidean metrics. If $\bY$ denotes a random variable with distribution $\P$  taking its value in a metric space $(\MM,d_{\MM})$, then a Fréchet mean (not necessarily unique) of the distribution $\P$  is a point $m^{\ast} \in \MM$ that is a global minimum (if any) of the functional 
$$
J(m) = \frac{1}{2} \int_{\MM} d^{2}_{\MM}(m,y) d\P(y) \quad \mbox{ i.e. }  \quad  m^{\ast} \in \argmin_{m \in \MM} J(m).
$$
In this paper, a Fréchet mean of a random variable $\bY$ with distribution $\P$ will be also called a barycenter. An  empirical Fréchet mean of an   independent and  identically distributed (iid) sample $\bY_{1},\ldots,\bY_{n}$ of distribution $\P$ is
$$
\bar{\bY}_{n} \in \argmin_{m \in \MM} \frac{1}{n} \sum_{j=1}^{n}  \frac{1}{2} d^{2}_{\MM}(m,\bY_{j}).
$$
 For random variables belonging to nonlinear metric spaces, a well-known example is the computation of  the mean of a set of planar shapes in the Kendall's shape space \cite{kendall} that leads to the  Procrustean means studied in  \cite{MR1108330}. Many properties of the Fr\'echet mean in finite dimensional Riemannian manifolds (such as consistency and uniqueness) have been investigated in \cite{Afsari, MR2914758,batach1,batach2,MR2816349}. For random variables taking their value in metric spaces of nonpositive curvature (NPC), a detailed study of various properties of their barycenter  can be found in  \cite{MR2039961}.  However, there is not so much work on Fr\'echet means in infinite dimensional metric spaces that do not satisfy the global NPC property as defined in  \cite{MR2039961}. 
 
 \subsection{A parametric class of random probability measures}
 
 Let $\Omega  = \overline{B(0,r)} \subset \RR^{d}$ be the closed ball centered at zero of a given radius $r > 0$. In this paper, we consider the case where $\bY = \bmu$ is a random probability measure whose support is  included in $\Omega$. The support of the measure $\bmu$ is understood as the smallest closed set of $\bmu$-mass equal to 1.
 
 Let us now define  a specific parametric class of random probability measures. Let $\MM_{+}(\Omega)$ be  the set of  Radon probability measures with support included in $\Omega$ endowed with  the 2-Wasserstein distance $d_{W_{2}}$ between two probability measures. Let $\Phi : (\RR^{p}, \BB(\RR^{p})) \to (\MM_{+}(\Omega), \BB \left( \MM_{+}(\Omega) \right)$ be a measurable mapping, where $\BB(\RR^{p})$ is the Borel $\sigma$-algebra of $\RR^{p}$ and $\BB \left( \MM_{+}(\Omega) \right)$ is the  Borel $\sigma$-algebra  generated by the topology induced by the distance $d_{W_{2}}$. For a measurable subset  $\Theta$ of $\RR^{p}$ (with $p \geq 1$), we consider the parametric set of probability measures $\{ \mu_{\theta} = \Phi(\theta), \; \theta \in \Theta \}$. Futhermore, we assume that, for any $\theta \in \Theta$, the measure $\mu_{\theta} = \Phi(\theta) \in \MM_{+}(\Omega)$ admits a density with respect to the Lebesgue measure on $\RR^{d}$. It follows that if $\btheta \in \RR^{p}$ is a random vector with distribution $\P_{\Theta}$ admitting a density $g : \Theta \to \RR_{+}$, then $\mu_{\btheta} = \Phi(\btheta)$ is a random probability measure with distribution $\P_{g}$ on $(\MM_{+}(\Omega), \BB \left( \MM_{+}(\Omega) \right)$ that is the push-forward measure defined by
\begin{equation} \label{eq:Pg}
\P_{g}(B) = \P_{\Theta}(\Phi^{-1}(B)), \mbox{ for any } B \in  \BB \left( \MM_{+}(\Omega) \right).
\end{equation}
In this paper, we propose to study some properties of the barycenter $\mu^{\ast}$ of $\mu_{\btheta}$ defined as the following Fréchet mean
\begin{eqnarray}
\mu^{\ast} & = &  \argmin_{\nu \in \MM_{+}(\Omega)}  J(\nu), \label{eq:pbgen} 
\end{eqnarray}
where
$$
J(\nu) = \int_{ \MM_{+}(\Omega)}  \frac{1}{2} d^{2}_{W_{2}}(\nu, \mu )  d \P_{g}(\mu) = \EE \left( \frac{1}{2}  d^{2}_{W_{2}}(\nu, \mu_{\btheta}) \right) = \int_{\Theta}  \frac{1}{2} d^{2}_{W_{2}}(\nu, \mu_{\theta}) g(\theta) d \theta, \; \nu \in  \MM_{+}(\Omega). 
$$
 If it exists and is unique,  the measure $\mu^{\ast}$ will be referred to as the population barycenter of the random measure $\mu_{\btheta}$ with distribution $\P_{g}$.
 
 The empirical counterpart  of $\mu^{\ast}$ is the barycenter $\bar{\bmu}_{n}$ defined as
\begin{eqnarray}
\bar{\bmu}_{n} & = &  \argmin_{\nu \in \MM_{+}(\Omega) }  \frac{1}{n} \sum_{i=1}^{n} \frac{1}{2} d^{2}_{W_{2}}(\nu, \mu_{\btheta_{i}} ),  \label{eq:empbary} 
\end{eqnarray}
where $\btheta_{1},\ldots,\btheta_{n}$ are iid random  vectors in $\Theta$ with density $g$.  

\subsection{Main results of the paper}

The main contribution of this paper is to give an explicit characterization of the population barycenter $\mu^{\ast}$ that can be (informally)  stated as follows:  let $\mu_{0} \in \MM_{+}(\Omega)$ be a fixed reference measure that is absolutely continuous with respect to the Lebesgue measure on $\RR^d$, and let $T_{\btheta} : \Omega \to \Omega$ be the optimal mapping that transports $\mu_{0}$ onto $\mu_{\btheta}$. This mapping is such that $ \mu_{\btheta} = T_{\btheta} \# \mu_{0}$, where  $T_{\btheta} \# \mu_{0}$ denotes the push-forward of the measure  $\mu_{0}$, and its satisfies
$$
d^{2}_{W_{2}}(\mu_{0}, \mu_{\btheta} ) = \int_{\Omega} |T_{\btheta}(x) - x |^{2} d\mu_{0}(x).
$$
Thanks to the well-known Brenier's theorem \cite{Brenier91} the mapping $T_{\btheta}$ is uniquely defined $\mu_{0}$-almost everywhere on the support $\Omega_{0} \subset \Omega$ of the reference measure $\mu_{0}$. %In the paper, it will be needed to also consider $T_{\btheta}$ as a mapping from $\Omega$ to $\Omega$, by extending $T_{\btheta}$ outside the support $\Omega_{0}$ in a smooth way (further details on this extension are discussed later on).

A first result of this paper is that if $\EE(  T_{\btheta}(x) ) = x$ for all $x \in \Omega_{0}$, then $\mu_{0}$  is equal to  the population barycenter $\mu^{\ast}$, and one has that
$$
\inf_{\nu \in \MM_{+}(\Omega)}  J(\nu) =  \frac{1}{2}  \int_{\Omega}  \EE \left(  | T_{\btheta}(x)   - x |^{2} \right) d \mu^{\ast} (x) =  \frac{1}{2}  \int_{\Omega}  \int_{\Theta} \left(  | T_{\theta}(x)   - x |^{2} \right) g(\theta) d \theta d \mu^{\ast} (x).
$$
This property is already known for empirical barycenters from the arguments in Remark 3.9 in \cite{MR2801182}.
%provided that the ``variance'' of the random mapping $T_{\btheta}$ is finite, namely that
%$$
%\int_{\Omega} \EE \left(  | T_{\btheta}(x)   - x |^{2} \right) d \mu^{\ast} (x) < + \infty.
%$$
This characterization of barycenters can also be written in the form
$
\mu^{\ast} =  \bar{T} \# \mu_{0},
$
where $\bar{T} \ = \EE(  T_{\btheta} ) $ is the mapping defined by $\bar{T}(x) =  \EE(  T_{\btheta}(x) )$, for all $x \in \Omega_{0}$. This suggests that averaging in the  Wasserstein space may amount to take the expectation (in the usual sense) of the optimal transport map $T_{\btheta}$ with respect to a fixed reference measure $\mu_{0}$. However, this result is generally not true if $\EE(  T_{\btheta} )\neq I$ % $\mu^{\ast}$-almost everywhere
, where $I = \Omega_{0} \to \Omega_{0}$ denotes the identity mapping. Nevertheless, we  propose to consider sufficient conditions (beyond the case $\EE(  T_{\btheta} ) =  I$) ensuring that 
\begin{equation} \label{eq:expectmapping}
\mu^{\ast} =   \bar{T}   \# \mu_{0}. %, \mbox{ where } \bar{T} = \EE(  T_{\btheta} ).
\end{equation}
 To the best of our knowledge, such a result on the population barycenter in the  2-Wasserstein space is new. A similar result has been established  in \cite{MR3338645} for the empirical barycenter, namely that
\begin{equation} \label{eq:empmapping}
\bar{\bmu}_{n} =\left( \frac{1}{n} \sum_{i=1}^{n}  T_{\btheta_{i}} \right)  \# \mu_{0},
\end{equation}
under specific assumptions on the optimal maps $ T_{\btheta_{i}}$ and the reference measure $\mu_{0}$ such that $ \mu_{\btheta_{j}} = T_{\btheta_{j}} \# \mu_{0}$ for $i=1,\ldots,n$.  The validity of equation \eqref{eq:empmapping} stated in \cite{MR3338645}  is restricted to a specific class of optimal transport maps  that is assumed to be admissible in the following sense (see Definition 4.2 in \cite{MR3338645}): there exists $i_{0}$ such that $ T_{\btheta_{i_{0}}} = I$, for each $1 \leq i \leq n$, $T_{\btheta_{i}}$ is a one-to-one mapping,  and the composition $T_{\btheta_{i}} \circ T_{\btheta_{j}}^{-1}$ of two maps in this class remains an optimal one for all $1 \leq i,j \leq n$. %In this paper, we do not necessarly make  such assumptions on the composition $T_{\btheta} \circ T_{\btheta'}^{-1}$ of two maps for $\btheta \neq \btheta'$. 
The extension of the results in \cite{MR3338645}  to the population barycenter has not been considered so far. In this paper, when $\bar{T}  = \EE(  T_{\btheta} ) \neq I$, we do not assume that the collection of maps  $(T_{\theta})_{\theta \in \Theta}$ is an admissible class to prove that $\mu^{\ast} =  \bar{T}   \# \mu_{0}$, in the sense that it is not required that $T_{\theta} \circ T_{\theta'}^{-1}$ is an optimal map for any  $\theta, \theta' \neq \Theta$. 
Indeed,  to prove equation \eqref{eq:expectmapping}, we mainly use, in this work,  the assumption that $T_{\theta} \circ \bar{T}^{-1}$ is an optimal map for all $\theta\in \Theta$. For some deformable models of signals and images, it will be shown that such an assumption may be weaker than the notion of admissible maps introduced in   \cite{MR3338645}.

Equation \eqref{eq:expectmapping}  is not difficult to prove in a one-dimensional setting ($d=1$) i.e.\ for random measures supported on the real line. The main issue is then to extend this result to higher dimensions $d \geq 2$. To this end, we consider a dual formulation of the optimization problem \eqref{eq:pbgen} that allows a precise characterization of some properties of the population barycenter that are used to prove equation  \eqref{eq:expectmapping}.  These results are based on an adaptation of the duality arguments developed in \cite{MR2801182} for the characterization of an empirical barycenter.  Therefore, our approach is very much  connected with the theory of optimal mass transport, and with the characterization of the Monge-Kantorovich problem via arguments from convex analysis and duality, see \cite{villani2003tot} for further details on this topic.

Another contribution of this paper is to discuss the usefulness of barycenters in the Wasserstein space for the statistical analysis of deformable models in signal and image processing. Statistical deformable models are widely used for the analysis of a sequence of random signals or images showing a significant amount of geometric variability in time or space. We refer to \cite{ABGMR13} for a detailed review of deformable models. In such settings, it is of fundamental interest to propose a consistent notion of  averaging for  signals or images sampled from such models. In this paper, we study random measures $\mu_{\btheta}$ whose associated densities $q_{\btheta}$ satisfy the following semi-parametric deformable model: 
\begin{equation}
q_{\btheta}(x)  = \left| \det\left( D \varphi_{\btheta}^{-1} \right) (x) \right|  q_{0}\left(  \varphi_{\btheta}^{-1}(x) \right) \; x \in \Omega,   \label{eq:modeldeform}
\end{equation}
where $\varphi_{\btheta} : \Omega \to \Omega$  is a random parametric diffeomorphism, $\det\left( D \varphi_{\btheta}^{-1} \right) (x)$ denotes the determinant of its Jacobian matrix at point $x$, and $q_{0}$ is a fixed reference density with support $\Omega_{0} \subset \Omega$. In model \eqref{eq:modeldeform}, it  seems natural to define the ``average of  $q_{\btheta}$'' as the density $\bar{q}$ given by
$$
\bar{q}(x)  = \left| \det\left( D \bar{\varphi}^{-1} \right) (x) \right|  q_{0}\left( \bar{\varphi}^{-1}(x) \right) \; x \in \Omega,  
$$
 where $\bar{\varphi} : \Omega_{0} \to \Omega$ is defined by $\bar{\varphi}(x) = \EE \left(  \varphi_{\btheta}(x)  \right) = \int_{\Theta}  \varphi_{\theta}(x)  g(\theta) d \theta, \; x \in \Omega_{0}$. In this paper, a novel result  is to propose  sufficient conditions on the random diffeomorphism $\varphi_{\btheta}$ which ensure that the measure $\bar{\mu}$ with density $\bar{q}$ corresponds to the population barycenter in the Wasserstein space of the random measure $\mu_{\btheta}$.

Finally, we also study the consistency of the empirical barycenter  $\bar{\bmu}_{n}$ to its population counterpart $\mu^{\ast}$ as the number $n$ of measures tends to infinity.

\subsection{Related results in the literature} \label{sec:litt}

%Throughout the paper, we discuss various statistical models that  involve the use of such parametric random measures, and for which the notion of barycenter in the 2-Wasserstein space is relevant. 
A similar notion (to the one in this paper) of a population barycenter and its connection  to optimal transportation with infinitely many marginals have  been  studied in \cite{MR3004954}.  In particular, a similar class of parametric random probability measures, where the parameter set $\Theta$ is one-dimensional, is also considered in \cite{MR3004954} for the purpose of studying the existence and uniqueness of $\mu^{\ast}$. Some generalization of these notions for probability measures defined on a Riemannian manifold, equipped with the Wasserstein metric, have been recently proposed in \cite{legouic:hal-01163262} and \cite{KimPass}. A detailed characterization of empirical barycenters, in a broader setting, in terms of existence, uniqueness and regularity, together with its link to the multi-marginal problem  in optimal transport can be found in \cite{MR2801182}.

 %{\color{blue}
 Our work extends some of the results in \cite{MR2801182} on the Wasserstein barycenter of a finite set of measures to the case of a population barycenter for a general parametric distribution $\P_{g}$ on the set of probability measures $\MM_{+}(\Omega)$. In particular, we extend the dual characterization of empirical Wasserstein barycenters as proposed in \cite{MR2801182} to obtain a dual formation of the optimisation problem \eqref{eq:pbgen}.  The results in this paper are also connected to the recent work \cite{fixed-point} which proposes an iterative algorithm to compute an empirical Wasserstein barycenter which consists in iterating the following two steps:
\begin{description}
\item[-] choice of a reference measure onto which the observed measures are mapped,
\item[-] pusforward of the reference measure by the average of these maps, 
\end{description}
until obtaining a reference measure which is a Wasserstein barycenter. In this paper, we somewhat discuss conditions for which such an algorithm  converges in one iteration to obtain a population barycenter.
%}

In the literature on signal and image processing, there exists various  applications of the notion of an empirical barycenter in the Wasserstein space. For example, it has been successfully used  for texture analysis in image processing \cite{2013-Bonneel-barycenter,rabin-ssvm-11}. There also exists a growing interest on the development of fast algorithms for the computation of empirical barycenters with various applications in image processing \cite{2014-Benamou,cuturi-14}. The theory of optimal transport for image warping has also been shown to be usefull tool, see e.g.\ \cite{MR2058553, Haker04optimalmass} and references therein.  Some properties of the empirical barycenter   in the 2-Wasserstein space of random measures satisfying a deformable model similar to \eqref{eq:modeldeform} have also been studied in \cite{MR3338645}.  To study the registration problem of distributions,  the use of semi-parametric models of densities similar to \eqref{eq:modeldeform} has also been considered in \cite{ACLL15}. 

The main contribution of this paper, with respect to existing results in the literature, is to show the benefits of considering the dual formulation  of the (primal) problem \eqref{eq:pbgen} to characterize the population barycenter in the 2-Wasserstein space for a large class of deformable models of measures/densities. To the best of our knowledge, the characterization of a population barycenter in deformable models throughout such arguments is novel. 

\subsection{Organisation of the paper}

The paper is then organized as follows. In Section \ref{sec:gen}, we introduce some definitions and notation for the framework of the paper, and we discuss  the existence and uniqueness of the population barycenter. In Section \ref{sec:realline}, we characterize the population barycenter in the one-dimensional case, i.e.\ for random measures supported on $\Omega \subset \RR$. In Section \ref{sec:dual}, we study a dual formulation of the optimization problem \eqref{eq:pbgen}.  In Section \ref{sec:main}, we prove the main result of the paper, namely equation \eqref{eq:expectmapping} in dimension $d \geq 2$. As an application of the methodology developed in this paper, we discuss, in Section \ref{sec:appli}, the usefulness of barycenters in the Wasserstein  space for the analysis of deformable models in statistics. The consistency of the empirical barycenter is discussed in  Section \ref{sec:empbary}.   In Section \ref{sec:persp} , we  give some perspectives on the extension of this work.

\section{Existence and uniqueness of the population barycenter} \label{sec:gen}
\subsection{Some definitions and notation}

We use bold symbols $\bY, \bmu, \btheta, \ldots$ to denote random objects. The notation $|x|$ is used to denote the usual Euclidean norm of a vector $x \in \RR^{m}$ and the notation $\langle x , y \rangle$ denotes the usual inner product for $x,y \in \RR^{m}$. We recall that $\MM_{+}(\Omega)$ is  the set of  Radon probability measures with support  included in $\Omega = \overline{B(0,r)}$, and that the following assumption is made throughout the paper.

%\begin{ass} \label{ass:Compact}
%The set $\Theta \subset \RR^{p}$ is compact.
%\end{ass}

\begin{ass} \label{ass:Phi}
The mapping $\Phi : (\Theta\subset\RR^{p}, \BB(\RR^{p})) \to (\MM_{+}(\Omega), \BB \left( \MM_{+}(\Omega) \right)$ is measurable. Moreover, it is such that, for any $\theta \in \Theta$,  the measure $\mu_{\theta} = \Phi(\theta) \in \MM_{+}(\Omega)$ admits a density with respect to the Lebesgue measure on $\RR^{d}$.
\end{ass}

The squared 2-Wasserstein distance between two probability measures $\mu, \nu \in\MM_{+}(\Omega)$ is 
$$
d^{2}_{W_{2}}(\mu, \nu):=\inf_{  \gamma\in\Pi(\mu,\nu) } \left\{\int_{\Omega\times\Omega}|x-y|^{2}d\gamma(x,y) \right\},
$$
where $\Pi(\mu,\nu)$ is the set of all probability measures on $\Omega \times \Omega$ having $\mu$ and $\nu$ as marginals, see e.g.\ \cite{villani2003tot}.
We recall that $\tilde{\gamma} \in \Pi(\mu,\nu)$ is called an optimal transport plan between $\mu$ and $\nu$ if
$$
d^{2}_{W_{2}}(\mu, \nu) = \int_{\Omega\times\Omega}|x-y|^{2}d\tilde{\gamma}(x,y).
$$
Let $T : \Omega \to \Omega$ be a measurable mapping, and let $\mu \in \MM_{+}(\Omega)$. The push-forward measure $ T  \# \mu $ of $\mu$ through the map $T$ is the measure defined by duality as
$$
\int_{\Omega} f(x) d( T  \# \mu  )(x) =  \int_{\Omega} f( T(x) ) d \mu(x), \; \mbox{ for all continuous and bounded functions } f:\Omega\longrightarrow\RR.
$$
We also recall the following well known result  in optimal transport due to Brenier \cite{Brenier91} (see also \cite{villani2003tot} or Proposition 3.3 in \cite{MR2801182}):
\begin{prop} \label{prop:Brenier}
Let $\mu, \nu \in\MM_{+}(\Omega)$. Then, $\gamma \in \Pi(\mu,\nu)$ is an optimal transport plan between  $\mu$ and $\nu$ if and only if the support of $\gamma$ is included in  the set $\partial \phi$ that is the graph of the subdifferential of a convex and  lower semi-continuous function $\phi$ satisfying
$$
\phi =  \argmin_{\psi \; \in \; \mathcal{C} } \left\{ \int_{\Omega} \psi(x) d \mu(x) + \int_{\Omega} \psi^{\ast}(x) d \nu(x) \right\},
$$
where $\psi^{\ast}(x) = \sup_{y \in \Omega} \left\{ \langle x,y \rangle - \psi(y) \right\} $ is the convex conjugate of $\psi$, and $\mathcal{C}$ denotes the set of proper convex functions $\psi : \Omega \to \RR$ that are lower semi-continuous 

If $\mu$ admits a density with respect to the Lebesgue measure on $\RR^{d}$, then there exists a unique optimal transport plan $\gamma \in \Pi(\mu,\nu)$ that is of the form $\gamma = (id, T) \# \mu$ where $T = \nabla \phi$ (the gradient of  $\phi$)  is called the optimal mapping between $\mu$ and $\nu$. The uniqueness of the transport plan holds in the sense that if $\nabla \phi \# \mu = \nabla \psi \# \mu$, where $\psi : \Omega \to \RR$ is a lower semi-continuous convex function, then  $\nabla \phi = \nabla \psi$ $\mu$-almost everywhere. Moreover, one has that
$$
d^{2}_{W_{2}}(\mu, \nu) = \int_{\Omega} | \nabla \phi(x) - x|^{2} d \mu(x).
$$
\end{prop}

\subsection{About the measurability of  $T_{\theta}$} \label{subsec:T}

 Let $\mu_{0}$ be a fixed reference measure that is absolutely continuous with respect to the Lebesgue measure on $\RR^d$. In the introduction, for any $\theta \in \Theta$ , we have defined $T_{\theta} : \Omega \to \Omega$ as the optimal mapping to transport $\mu_{0}$ onto $\mu_{\theta}$. By Proposition \ref{prop:Brenier}, such a mapping  exists, and it is the  $\mu_{0}$-almost everywhere unique one that can be written as the gradient of a convex function $\phi_{\theta}$. Then, thanks to Assumption \ref{ass:Phi}, it follows, by Theorem 1.1 in \cite{MR2643592}, that there exists a function $(\theta,x) \mapsto T(\theta,x)$ which is measurable with respect to the $\sigma$-algebra $  \BB(\RR^{p})  \otimes \BB \left( \MM_{+}(\Omega) \right)$, and such that for $\P_{\Theta}$-almost everywhere $\theta \in \Theta$,
$$
T(\theta,x) =  T_{\theta}(x), \; \mbox{ for } \mu_{0}\text{-almost everywhere } x \in \Omega.
$$
In particular, the  mapping $(x,\theta) \mapsto T_{\theta}(x)$ is measurable with respect to the completion of $  \BB(\RR^{p})  \otimes \BB \left( \MM_{+}(\Omega) \right)$ with respect to $g(\theta) d\theta d \mu_{0}(x)$ (we recall  the assumption that $d \P_{\Theta}(\theta) =   g(\theta) d\theta$).
 
 \subsection{Existence and uniqueness of $\mu^{\ast}$ }\label{sec:exist}
 
 Let us now prove the existence and uniqueness of the barycenter (i.e.\ the Fréchet mean) of the parametric random measures $\mu_{\btheta}$ with distribution $\P_{g}$, as defined in equation \eqref{eq:Pg}. We recall that this amounts to study the solution (if any) of the optimization problem \eqref{eq:pbgen}.

\begin{rem}[About the existence and finiteness of $J(\nu)$]\label{rem:finite} Let $\nu \in \MM_{+}(\Omega)$. The integral defining $J(\nu)$ will be defined as soon as $\theta \mapsto d^{2}_{W_{2}}(\nu, \mu_{\theta}) g(\theta)$ is a measurable application. The measurability of this mapping follows from Assumption  \ref{ass:Phi}. Moreover, since  $\Omega$ is compact, it follows that, for any $\theta \in \Theta$,  $d^{2}_{W_{2}}(\nu, \mu_{\theta}) \leq 4 \delta^{2}(\Omega)$, where $\delta^2(\Omega) = \sup_{x \in \Omega} \{ |x|^2\} < + \infty$. Therefore, one has that
$
J(\nu) = \frac{1}{2} \int_{\Theta} d^{2}_{W_{2}}(\nu, \mu_{\theta}) g(\theta) d \theta \leq 2 \delta^{2}(\Omega) < + \infty,
$
for any $\nu \in \MM_{+}(\Omega)$.
\end{rem} 
 
 As shown below, using standard arguments, the existence and the uniqueness of  $\mu^{\ast}$ are not difficult to prove. To this end, a key property of the functional $J$ defined in \eqref{eq:pbgen} is the following:
\begin{lem} \label{lem:convex}
Suppose that  Assumption  \ref{ass:Phi} holds. Then, the functional $J : \MM_{+}(\Omega) \to \RR$ is strictly convex in the sense that
\begin{equation}
J(\lambda \mu + (1-\lambda) \nu ) < \lambda J(\mu)  + (1-\lambda) J(\nu), \mbox{ for any } \lambda \in ]0,1[ \mbox{ and } \mu, \nu \in \MM_{+}(\Omega) \mbox{ with } \mu \neq \nu.   \label{eq:convexJ}
\end{equation}
\end{lem}
\begin{proof} 
Inequality \eqref{eq:convexJ} follows immediately from the assumption that, for any $\theta \in \Theta$, the measure $\mu_{\theta}$ admits a density with respect to the Lebesgue measure on $\RR^{d}$, and from the use of Theorem 2.9 in \cite{MR2814414} (for similar results  to those in  \cite{MR2814414} on the strict convexity, one can read Lemma 3.2.1 in \cite{MR3004954}). 
\end{proof}
Hence, thanks to the strict convexity of $J$,   if a barycenter $\mu^{\ast}$ exists, then it is necessarily unique. The existence of $\mu^{\ast}$ is then proved in the next proposition.

\begin{prop}  \label{prop:exist} Under  Assumption  \ref{ass:Phi}, the optimization problem   \eqref{eq:pbgen}  admits a unique minimizer. 
\end{prop}

\begin{proof}

 \noindent Let $\nu^{n}$ be a minimizing sequence of the optimization problem  \eqref{eq:pbgen}. Since $\Omega$ is compact, the sequence $\int_{\Omega} |x|^{2} d\nu^{n}(x)$ is uniformly bounded by $\delta^{2}(\Omega)$. Hence, by Chebyshev's inequality, the sequence $\nu_{n}$ is tight and by Prokhorov's Theorem there exists a (non relabeled) subsequence that weakly converges to some $\mu^{\ast} \in \MM_{+}(\Omega)$. % Since $\Omega$ is compact, it can be checked that $\nu \mapsto d^{2}_{W_{2}}(\mu_{\theta},\nu)$ is lower semi-continuous (l.s.c.) on $\MM_{+}(\Omega)$.
Therefore,
$ \frac{1}{2} d^{2}_{W_{2}}(\mu_{\theta},\mu^{\ast}) \leq   \liminf_{n \to + \infty}  \frac{1}{2} d^{2}_{W_{2}}(\mu_{\theta},\nu^{n}) $, % ({\bf A vérifier !}) 
and thus, by Fatou's Lemma
$$
\int_{\Theta} \frac{1}{2} d^{2}_{W_{2}}(\mu_{\theta},\mu^{\ast}) g(\theta) d \theta \leq  \int_{\Theta}  \liminf_{n \to + \infty}  \frac{1}{2} d^{2}_{W_{2}}(\mu_{\theta},\nu^{n}) g(\theta)   d \theta   \leq  \liminf_{n \to + \infty}    \int_{\Theta} \frac{1}{2} d^{2}_{W_{2}}(\mu_{\theta},\nu^{n}) g(\theta) d \theta.
$$
Therefore, $J(\mu^{\ast}) =  \inf_{\nu \in \MM_{+}(\Omega)}  \frac{1}{2} \int_{\Theta} d^{2}_{W_{2}}(\nu, \mu_{\theta}) g(\theta) d \theta$, which proves that the optimization problem  \eqref{eq:pbgen} admits a minimizer. 

By the strict convexity of  the functional $J$ (as stated in Lemma \ref{lem:convex}), it follows that the barycenter of $\mu_{\btheta}$   is  necessarily unique.
\end{proof}

%%%%%%%%%%%%%%%%%%%%%%%%%%%%%%%%%%%%%%%%%%

\section{Barycenter for measures supported  on the real line} \label{sec:realline}

In this section, we study some properties of the population barycenter $\mu^{\ast}$ of random measures supported on the real line i.e.\  we consider the case $d=1$ where $\Omega = [-r,r] \subset \RR$. In this setting, our characterization of $\mu^{\ast}$ will follow from the well known fact that if $\mu$ and $\nu$ are measures belonging to $\MM_{+}(\Omega)$ then
\begin{equation}
d^{2}_{W_{2}}(\nu, \mu )=\int_{0}^{1}\left|F_{\nu}^{-1}(x)-F_{\mu}^{-1}(x)\right|^{2}dx,  \label{eq:quantW2}
\end{equation}
where $F_{\nu}^{-1}$ (resp. $F_{\mu}^{-1}$) is the quantile function of $\nu$ (resp. $\mu$). This explicit expression for the Wasserstein distance allows a simple characterization of the barycenter of random measures. In particular, we prove that equation \eqref{eq:expectmapping} always holds for $d=1$. This result means that, in dimension 1, computing a barycenter in the Wasserstein space amounts to take the expectation (in the usual sense) of the optimal mapping to transport  a fixed (non-random) reference measure $\mu_{0}$ onto $\mu_{\btheta}$.

\begin{theo}\label{theo:1D}
Let  $\mu_{0}$ be any  fixed measure in $\MM_{+}(\Omega)$ that is absolutely continuous with respect to the Lebesgue measure, and whose supported is denoted by $\Omega_{0}$. %Suppose that  Assumption \ref{ass:Compact}
Suppose that  Assumption  \ref{ass:Phi} hold. Let $\mu_{\btheta} = \Phi(\btheta)$ be a random measure where $\btheta$ is a random vector in $\Theta$ with density $g$. Let $T_{\btheta}$  be the random optimal  mapping between $\mu_{0}$ and $\mu_{\btheta}$ given by $T_{\btheta}(x)=F^{-1}_{\mu_{\btheta}}\left(F_{\mu_{0}}(x)\right)$,  $x\in\Omega_{0}$, where $F^{-1}_{\mu_{\btheta}}$ is the quantile function of $\mu_{\btheta}$, and $F_{\mu_{0}}$ is the cumulative distribution function of $\mu_{0}$.

 Then, the barycenter of $\mu_{\btheta}$ exists,  is unique, and  satisfies:
\begin{equation}\label{eq:mu0}
\mu^{\ast}= \bar{T} \#\mu_{0},
\end{equation}
where $\bar{T} : \Omega_{0} \to \Omega$ denotes the optimal mapping between $\mu_{0}$ and $\mu^{\ast}$ that is  defined by
$$
\bar{T}(x) = \EE\left(T_{\btheta}(x)\right) = \int_{\Theta} T_{\theta}(x) g(\theta) d \theta, \mbox{ for all } x \in \Omega_{0}.
$$
Furthermore, the quantile function of $\mu^{\ast}$ is for $y \in [0,1]$
$$
F^{-1}_{\mu^{\ast}}(y)=\EE\left(F^{-1}_{\mu_{\btheta}}(y)\right) = \int_{\Theta} F^{-1}_{\mu_{\theta}}(y) g(\theta) d \theta.
$$
Thus, the barycenter $\mu^{\ast}$ does not depend on the choice of $\mu_{0}$, and one has that
\begin{equation}\label{eq:var1D}
\inf_{\nu \in \MM_{+}(\Omega)}  J(\nu) =  \frac{1}{2}  \int_{\Omega}  \EE \left(  | \bar{T}_{\btheta}(x)   - x |^{2} \right) d \mu^{\ast} (x) =  \frac{1}{2}  \int_{\Omega}   \int_{\Theta} | \bar{T}_{\theta}(x)   - x |^{2} g(\theta) d \theta d \mu^{\ast} (x) ,
\end{equation}
where $\bar{T}_{\btheta} = T_{\btheta} \circ \bar{T}^{-1}$.
\end{theo}
\begin{proof} 
Let $\nu\in \MM_{+}(\Omega)$ then 
$$
J(\nu) = \int_{\Theta} \frac{1}{2} d^{2}_{W_{2}}(\nu, \mu_{\theta} )g(\theta)d\theta= \frac{1}{2}  \int_{\Theta}\int_{0}^{1}\left|F_{\nu}^{-1}(y)-F_{\mu_{\theta}}^{-1}(y)\right|^{2}dy g(\theta)d\theta.
$$
In the proof, we will repeatedly use Fubini's Theorem to interchange the order of integration in the right-hand size of the above equation. To be valid, Fubini's Theorem requires that the application $(y,\theta)\mapsto\left|F_{\nu}^{-1}(y)-F_{\mu_{\theta}}^{-1}(y)\right|^{2} g(\theta)$ is measurable. This application will be measurable as soon as $(y,\theta)\mapsto F_{\mu_{\theta}}^{-1}(y)$ is measurable. Since, by definition, $F_{\mu_{\theta}}^{-1}(y) = T_{\theta}(F_{\mu_{0}}^{-1}(y))$ for any $y \in [0,1]$, the measurability of $(y,\theta)\mapsto F_{\mu_{\theta}}^{-1}(y)$ is ensured by the continuity of $y \mapsto F_{\mu_{0}}^{-1}(y)$ and the measurability of the mapping $(x,\theta) \mapsto T_{\theta}(x)$ which has been stated in Section \ref{subsec:T}.

%In order to prove this measurability let $t$ be a real number and consider
%$$
%A_{t}=\{(y,\theta),\ F_{\mu_{\theta}}^{-1}(y)\leq t\}=\{(y,\theta),\ y\leq F_{\mu_{\theta}}(t)\}.
%$$
%Hence, we can write that 
%$A_{t}=G_{t}^{-1}\left(]0,+\infty\right)$ 
%with $G_{t}(\theta,y)= F_{\mu_{\theta}}(t)-y$. Thanks to  Assumption  \ref{ass:Phi}, one has that $\theta\mapsto\mu_{\theta}$ is continuous for the topology induced by the Wasserstein distance, which implies that $\mu_{\theta_n}$ weakly converges to $\mu_{\theta}$ ({\color{red} {\bf A vérifier ?}}) when $\theta_n$ tends to $\theta$. Hence  $G_{t}$ is a continuous mapping. This property is enough to prove that $A_{t}$ is a measurable set, which implies the measurability of  $(y,\theta)\mapsto\left|F_{\nu}^{-1}(y)-F_{\mu_{\theta}}^{-1}(y)\right|^{2} g(\theta)$ .

Let us now consider the  function $\EE\left(F^{-1}_{\mu_{\btheta}}\right) : [0,1] \to \Omega$ defined by
$$
\EE\left(F^{-1}_{\mu_{\btheta}}\right)(y) := \EE\left(F^{-1}_{\mu_{\btheta}}(y)\right) = \int_{\Theta} F^{-1}_{\mu_{\theta}}(y) g(\theta) d \theta < + \infty
$$
for any $y \in [0,1]$. We will now prove that $\EE\left(F^{-1}_{\mu_{\btheta}}\right)$ is a quantile function. To this end, let $\delta^2(\Omega) = \sup_{x \in \Omega} \{ |x|^2\} = r^2$.  % Then, by Assumption  \ref{ass:Phi}, the measure $\mu_{\theta}$ admits a density with respect to the Lebesgue measure, and thus the function $y \mapsto F^{-1}_{\mu_{\theta}}(y)$ is continuous on $[0,1]$, for any $\theta \in \Theta$.
 %{\color{blue}
 Since $F^{-1}_{\mu_{\theta}}$ is a quantile function, it follows that the function $y \mapsto F^{-1}_{\mu_{\theta}}(y)$ is left-continuous on $(0,1)$, for any $\theta \in \Theta$. Hence, since  $|F^{-1}_{\mu_{\theta}}(y)| \leq  \delta(\Omega)$ for any $(\theta,y) \in \Theta \times [0,1]$ it follows, by Lebesgue's dominated convergence theorem, that $y \mapsto \EE\left(F^{-1}_{\mu_{\btheta}}\right)(y)$ is a left-continuous function on $(0,1)$. Since $y \mapsto F^{-1}_{\mu_{\theta}}(y)$ is non-decreasing for any $\theta \in \Theta$, it is also clear that $y \mapsto \EE\left(F^{-1}_{\mu_{\btheta}}\right)(y)$ is a non-decreasing function. Hence, thanks to the property that any left-continuous and non-decreasing function from $[0,1]$ to  $\Omega$ is the quantile of some probability measure belonging to $\MM_{+}(\Omega)$ (see e.g.\ Proposition A.2 in \cite{W1}), one obtains that $\EE\left(F^{-1}_{\mu_{\btheta}} \right)$ is a quantile function.
 %}
 
Now, it is clear that   $\EE\left(|F^{-1}_{\mu_{\btheta}}(y)|^2\right)  \leq  \delta^2(\Omega) < + \infty$ for any $y \in [0,1]$. Hence, applying Fubini's Theorem, and the fact that  $\EE\left|a-X\right|^{2}\geq\EE\left|\EE(X)-X\right|^{2}$ for any squared integrable real random variable $X$ and real number $a$, we obtain that
\begin{eqnarray}
\int_{\Theta}d^{2}_{W_{2}}(\nu, \mu_{\theta} )g(\theta)d\theta&=& \int_{0}^{1}\int_{\Theta}\left|F_{\nu}^{-1}(y)-F_{\mu_{\theta}}^{-1}(y)\right|^{2}g(\theta)d\theta dy = \int_{0}^{1} \EE  \left|F_{\nu}^{-1}(y)-F_{\mu_{\btheta}}^{-1}(y)\right|^{2}  dy  \nonumber \\
& \geq& \int_{0}^{1} \EE  \left| \EE\left(F^{-1}_{\mu_{\btheta}}(y)\right)-F_{\mu_{\btheta}}^{-1}(y)\right|^{2}  dy \nonumber \\
&=& \int_{0}^{1}\int_{\Theta}\left|\EE\left(F^{-1}_{\mu_{\btheta}}(y)\right)-F_{\mu_{\theta}}^{-1}(y)\right|^{2}g(\theta)d\theta dy \nonumber \\ 
& = &\int_{\Theta}d^{2}_{W_{2}}(\mu^{\ast}, \mu_{\theta} )g(\theta)d\theta, \label{eq:minJ1D}
\end{eqnarray}
where $\mu^{\ast}$ is the measure  in $\MM_{+}(\Omega)$ with quantile function given by  $F^{-1}_{\mu^{\ast}}=\EE\left(F^{-1}_{\mu_{\btheta}}\right)$. The above inequality shows that $J(\nu)  \geq J(\mu^{\ast}) $ for any $\nu\in \MM_{+}(\Omega)$. Therefore, $\mu^{\ast}$ is a barycenter of the random measure $\mu_{\btheta}$, and  the unicity of $\mu^{\ast}$ follows from the strict convexity of the  functional $J$ as stated in Lemma \ref{lem:convex}.
Finally, let  $\mu_{0}$ be any  fixed measure in $\MM_{+}(\Omega)$  that is absolutely continuous with respect to the Lebesgue measure, and whose support is $\Omega_{0}$. Hence, one has that $ F_{\mu_{0}}\circ F^{-1}_{\mu_{0}}(y) = y$ for any $y \in [0,1]$. Therefore, equation (\ref{eq:mu0}) follows from the  equalities
$$
F^{-1}_{\mu^{\ast}}=\EE\left(F^{-1}_{\mu_{\btheta}}\right)=\EE\left(F^{-1}_{\mu_{\btheta}}\circ F_{\mu_{0}} \right) \circ F^{-1}_{\mu_{0}} =  \EE\left(T_{\btheta}\right) \circ F^{-1}_{\mu_{0}} = \bar{T}\circ F^{-1}_{\mu_{0}}, 
$$
where $\bar{T} =  \EE\left(T_{\btheta}\right) = F^{-1}_{\mu^{\ast}} \circ  F_{\mu_{0}}$. Note that it is clear that $\bar{T} : \Omega_{0} \to \Omega$ is an increasing and continuous function, and thus it is the optimal mapping between $\mu_{0}$ and $\mu^{\ast}$ by Proposition \ref{prop:Brenier}. Finally,  by equation \eqref{eq:minJ1D} and Fubini's Theorem, one has that
\begin{eqnarray*}
\int_{\Theta}d^{2}_{W_{2}}(\mu^{\ast}, \mu_{\theta} )g(\theta)d\theta & = &  \int_{0}^{1}  \EE  \left| F^{-1}_{\mu^{\ast}}(y) -F_{\mu_{\btheta}}^{-1}(y)\right|^{2}  dy =  \EE  \left( \int_{0}^{1}  \left| F^{-1}_{\mu^{\ast}}(y) -F_{\mu_{\btheta}}^{-1}(y)\right|^{2}  dy \right) \\
& = &  \EE  \left( \int_{\Omega}  \left| F_{\mu_{\btheta}}^{-1} \circ F_{\mu^{\ast}}(x) -x\right|^{2}  d \mu^{\ast}(x) \right)  \\
& = &  \EE  \left( \int_{\Omega}  \left| F_{\mu_{\btheta}}^{-1} \circ F_{\mu_{0}} \circ  F_{\mu_{0}}^{-1} \circ F_{\mu^{\ast}}(x) -x\right|^{2}  d \mu^{\ast}(x) \right) \\
& = & \EE  \left( \int_{\Omega}  \left| T_{\btheta} \circ  \bar{T}^{-1}  (x) -x\right|^{2}  d \mu^{\ast}(x) \right) 
\end{eqnarray*}
where the last equality follows from the fact that $T_{\btheta} =F^{-1}_{\mu_{\btheta}}\circ F_{\mu_{0}}$ and $\bar{T}^{-1} =  F_{\mu_{0}}^{-1} \circ F_{\mu^{\ast}}$. Hence, this proves equation \eqref{eq:var1D}, and it completes the proof.  

%\QED
\end{proof}

To illustrate  Theorem \ref{theo:1D}, we consider   a  simple construction of  random probability measures in the case where $\Omega = [-r,r]$.  Let $\tilde{\mu} \in \MM_{+}^{2}(\Omega)$ admitting the density $\tilde{f}$ with respect to the Lebesgue measure on $\RR$, and cumulative distribution function (cdf) $\tilde{F}$. We assume that the density $\tilde{f}$ is continuous on $\Omega$, and that it is  supported in a sub-interval $\tilde{\Omega}$ of $\Omega$.  Let $\btheta = (\ba,\bb) \in ]0,+\infty[ \times \RR$ be a two dimensional random vector with density $g$, such that $\ba x+ \bb  \in \Omega$ for any $x \in \tilde{\Omega}$. We denote by $\mu_{\btheta}$   the  random probability measure admitting the density
$$
f_{\btheta}(x) = \frac{1}{\ba} \tilde{f} \left( \frac{x-\bb}{\ba}\right), \; x \in \Omega, %\label{eq:defnuab}
$$
where we have extended $\tilde{f}$ outside $\Omega$ by letting $\tilde{f}(u) = 0$ for $u \notin \Omega$. Thanks to our assumptions, the density $f_{\btheta}$ is  supported in a sub-interval of $\Omega$. Moreover, the cdf and quantile function of $\mu_{\btheta}$ are given by
$$
F_{\mu_{\btheta}}(x) = \tilde{F}\left( \frac{x-\bb}{\ba}\right),  \; x \in \Omega, \quad \mbox{ and }  \quad F_{\mu_{\btheta}}^{-1}(y) = \ba \tilde{F}^{-1}(y) + \bb, \; y \in [0,1]. % \label{eq:Fab}
$$
By Theorem \ref{theo:1D}, it follows that the barycenter of $\mu_{\btheta}$ is the probability measure  $\mu^{\ast}$  whose quantile function is given by
$$
F^{-1}_{\mu^{\ast}}(y) = \EE(\ba) \tilde{F}^{-1}(y) + \EE(\bb) , \; y \in [0,1]. 
$$
Therefore, $\mu^{\ast}$  admits the density
$$
f^{\ast}(x) = \frac{1}{\EE(\ba)} \tilde{f} \left( \frac{x-\EE(\bb)}{\E(\ba)}\right), x  \in \Omega,
$$
with respect to the Lebesgue measure on $\RR$.  Moreover, if $\mu_{0}$ is any  fixed measure in $\MM_{+}(\Omega)$, that is absolutely continuous with respect to the Lebesgue measure, then
$$
\mu^{\ast}= \bar{T} \#\mu_{0}, \quad \mbox{ where } \quad \bar{T}(x) = \EE(\ba) \tilde{F}^{-1}( F_{\mu_{0}}(x) ) + \EE(\bb), \; x \in \Omega_{0}.
$$

Now, we remark  that extending Theorem \ref{theo:1D} to dimension $d \geq 2$ is not straightforward. Indeed, the two key ingredients in the proof of  Theorem \ref{theo:1D} are the use of the well-known characterization \eqref{eq:quantW2} of the Wasserstein distance  in dimension $d=1$ via the quantile functions, and  the fact that, in dimension $d=1$,  the composition of two optimal maps (which are increasing functions) remains an optimal one (thanks to Proposition \ref{prop:Brenier}). However, the property \eqref{eq:quantW2}  which explicitly relates  the Wasserstein distance $d_{W_{2}}(\nu, \mu_{\theta} )$ to the marginal distributions of $\nu$ and $\mu_{\theta}$  is not valid in higher dimensions. Moreover, the composition of two optimal maps is not necessarily an optimal one in dimension $d \geq 2$. Nevertheless, we show in  Section \ref{sec:main} that  analogs of Theorem  \ref{theo:1D} can still be obtained in dimension $d \geq 2$.

%%%%%%%%%%%%%%%%%%%%%%%%%%%%%%%%%%%%%%%%%%

\section{Dual formulation} \label{sec:dual}

In this section, we introduce a dual formulation of problem  \eqref{eq:pbgen} that is inspired by the one proposed in \cite{MR2801182} to study the properties of empirical barycenters.  This dual formulation is then the key property to state the main result of this paper given in Section \ref{sec:main}. Let us recall the optimization problem \eqref{eq:pbgen}  as
\begin{eqnarray}
(\PP) \quad   J_{\PP}:= \inf_{\nu \in \MM_{+}(\Omega)} J(\nu),  \mbox{ where } J(\nu) = \frac{1}{2} \int_{\Theta} d^{2}_{W_{2}}(\nu, \mu_{\theta}) g(\theta) d \theta.  \label{eq:P} 
\end{eqnarray}
Then, let us introduce some definitions. Let $X = C(\Omega,\RR)$ be the space of continuous functions $f : \Omega \to \RR$ equipped with the supremum  norm
$$
\| f \|_{X} = \sup_{x \in \Omega} \left\{ |f(x)| \right\}.
$$
We also denote by $X' = \MM(\Omega)$  the topological dual of $X$, where $\MM(\Omega)$ is  the set of  Radon  measures with support included in $\Omega$. The notation $f^{\Theta} = ( f_{\theta } )_{\theta \in \Theta} \in L^{1}(\Theta, X) $  will denote any mapping
$$
\left\{
\begin{array}{cccc}
f^{\Theta} : & \Theta & \to & X    \\
 & \theta &  \mapsto &   f_{\theta}
\end{array}\right.
$$
such that
 %the function $(x,\theta) \mapsto  f_{\theta}(x)$ is continuous on $\Omega \times \Theta$. Note that assuming $f^{\Theta} \in L^{1}(\Theta, X)$ implies that
  for any $x \in \Omega$
$$
\int_{\Theta} |f_{\theta}(x)| d\theta < + \infty.
$$
Then, following the terminology in \cite{MR2801182}, we introduce the dual optimization problem
\begin{equation}
(\PP^{\ast}) \qquad  J_{\PP^{\ast}}:=\sup \left\{ \int_{\Theta} \int_{\Omega} S_{g(\theta)} f_{\theta}(x) d \mu_{\theta}(x)   d \theta; \; f^{\Theta} \in L^{1}(\Theta, X) \; \mbox{such that} \; \int_{\Theta} f_{\theta}(x)  d \theta = 0, \; \forall x \in \Omega \right\}, \label{eq:Past} 
\end{equation}
where
$$
S_{g(\theta)} f(x) := \inf_{y \in \Omega} \left\{ \frac{g(\theta)}{2} |x-y|^{2} - f(y) \right\}, \forall x \in \Omega \; \mbox{ and } \; f \in X.
$$
Let us also define 
$$
H_{g(\theta)}(f) := - \int_{\Omega} S_{g(\theta)} f(x) d \mu_{\theta}(x),
$$
and the Legendre-Fenchel transform of $H_{g(\theta)}$  for $\nu\in X'$ as
$$
H_{g(\theta)}^{\ast}(\nu) := \sup_{f \in X} \left\{ \int_{\Omega} f(x) d \nu(x) - H_{g(\theta)}(f)  \right\} = \sup_{f \in X} \left\{ \int_{\Omega} f(x) d \nu(x) + \int_{\Omega} S_{g(\theta)} f(x) d \mu_{\theta}(x)  \right\}.
$$
In what follows, we  show that the problems $(\PP)$ and $(\PP^{\ast})$ are dual to each other in the sense that the minimal value $J_{\PP}$ in problem $(\PP)$ is equal to the supremum $J_{\PP^{\ast}}$ in problem  $(\PP^{\ast})$. This duality is the key tool for the proof of Theorem \ref{prop:deform:measure}.

\begin{prop}  \label{prop:minmax} % Suppose that Assumption \ref{ass:Compact} and 
Suppose that  Assumption \ref{ass:Phi} is satisfied. Then,
$$J_{\PP} = J_{\PP^{\ast}} .$$
\end{prop}
\begin{proof} 

 \noindent 1. Let us first prove that  $J_{\PP} \geq J_{\PP^{\ast}} .$\\
By definition for any $f^{\Theta} \in L^{1}(\Theta, X)$ such that  $ \forall x\in\Omega, \  \int_{\Theta} f_{\theta}(x) d \theta = 0,$ and  for all $y\in\Omega$ we have
$$
S_{g(\theta)} f_{\theta}(x)+f_{\theta}(y)\leq  \frac{g(\theta)}{2} |x-y|^{2}.
$$
Let $\nu\in \MM_{+}(\Omega)$ and $\gamma_{\theta}\in\Pi(\mu_{\theta},\nu)$ be an optimal transport plan between $\mu_{\theta}$ and $\nu$. By integrating the above inequality with respect to $\gamma_{\theta}$ we obtain
 $$
\int_{\Omega}S_{g(\theta)} f_{\theta}(x)d\mu_{\theta}(x)+\int_{\Omega}f_{\theta}(y)d\nu(y)\leq \int_{\Omega\times\Omega} \frac{g(\theta)}{2} |x-y|^{2}d\gamma_{\theta}(x,y)= \frac{g(\theta)}{2} d_{W_{2}}^{2}(\mu_{\theta},\nu).
$$
Integrating now with respect to $d\theta$ and using Fubini's Theorem we get
 $$
\int_{\Theta}\int_{\Omega}S_{g(\theta)} f_{\theta}(x)d\mu_{\theta}(x)d\theta\leq \int_{\Theta} \frac{g(\theta)}{2} d_{W_{2}}^{2}(\mu_{\theta},\nu)d\theta.
$$
Therefore we deduce that  $J_{\PP} \geq J_{\PP^{\ast}} .$ \\

\noindent 2.  Let us now prove the converse inequality  $J_{\PP} \leq J_{\PP^{\ast}} .$\\
Thanks to the Kantorovich duality formula (see e.g.\ \cite{villani2003tot}, or Lemma 2.1 in \cite{MR2801182}) we have that $H_{g(\theta)}^{*}(\nu) = \frac{1}{2} d^{2}_{W_{2}}(\mu_{\theta},\nu) g(\theta)$ for any $\nu \in \MM_{+}(\Omega)$. Therefore, it follows that
\begin{equation}
J_P=\inf\left\{\int_{\Theta}H_{g(\theta)}^{*}(\nu)d\theta,\ \nu\in X'\right\}=-\left( \int_{\Theta}H_{g(\theta)}^{*}d\theta \right)^{*}(0). \label{eq:JPstar}
\end{equation}
Define the inf-convolution of $\left(H_{g(\theta)}\right)_{\theta\in\Theta}$ by
$$
H(f)  :=\inf\left\{ \int_{\Theta}H_{g(\theta)}(f_{\theta})d\theta;\ f^{\Theta} \in L^{1}(\Theta, X) , \int_{\Theta}f_\theta(x)d\theta=f(x), \forall x\in\Omega\right\},\quad\forall f \in X.
$$ 
We have in the other hand that
$$
J_{\PP^{\ast}}=-H(0).
$$
Using Theorem 1.6 in \cite{MR0288601}, one has that for any $\nu\in \MM_{+}(\Omega)$
$$
H^{\ast}(\nu)  =   \int_{\Theta} H^{\ast}_{g(\theta)}(\nu) d \theta.
$$
Then, thanks to \eqref{eq:JPstar}, it follows that
$$
J_P = -H^{\ast \ast}(0) \geq -H(0) = J_{\PP^{\ast}}.
$$
Let us now prove that $H^{\ast \ast}(0) = H(0)$. Since $H$ is convex it is sufficient to show that $H$ is continuous at 0 for the supremum norm of the space $X$ (see e.g.\ Proposition 4.1 in  \cite{MR1727362}).
%-  {\bf \color{red} Préciser le théorème auquel on fait référence ?})
 For this purpose, let $f^{\Theta} \in L^{1}(\Theta, X)$ and  remark that it follows from the definition of $H_{g(\theta)}$ that
\begin{eqnarray*}
H_{g(\theta)}(f_{\theta})  & = & \int_{\Omega}  \sup_{y \in \Omega} \left\{  f_{\theta}(y) - \frac{g(\theta)}{2} |x-y|^{2} \right\}  d \mu_{\theta}(x) \\
& \geq &   f_{\theta}(0) - \frac{g(\theta)}{2} \int_{\Omega}   |x|^{2}  d \mu_{\theta}(x),
\end{eqnarray*}
which implies that
$$
H(f) \geq f(0) - \int_{\Theta}  \frac{g(\theta)}{2} \int_{\Omega}   |x|^{2}  d \mu_{\theta}(x) d \theta > - \infty, \; \forall f \in X.
$$
Let $f \in X $ such that $\| f \|_{X} \leq 1/4$ and choose $f^{\Theta} \in L^{1}(\Theta, X)$ defined by $f_{\theta}(x) = f(x) g(\theta)$ for all $\theta \in \Theta$ and $x \in \Omega$. It follows that
\begin{eqnarray*}
H(f) & \leq & \int_{\Theta}H_{g(\theta)}(f(\cdot) g(\theta))d\theta \leq  \int_{\Theta} \int_{\Omega}  \sup_{y \in \Omega} \left\{  \frac{g(\theta)}{4}  - \frac{g(\theta)}{2} |x-y|^{2} \right\}  d \mu_{\theta}(x)  d\theta \\
& \leq & \int_{\Theta}   \int_{\Omega}    \frac{g(\theta)}{4}  d \mu_{\theta}(x)  d\theta =  \frac{1}{4}.
\end{eqnarray*}
Hence, the convex function $H$ never takes the value $- \infty$ and is bounded from above in a neighborhood of 0 in $X$. Therefore, by standard results in convex analysis (see e.g.\ Lemma 2.1 in  \cite{MR1727362}), 
%-  {\bf \color{red} Préciser le théorème auquel on fait référence ?})
$H$ is continuous at 0, and therefore $H^{\ast \ast}(0) = H(0)$ which completes the proof. 
\end{proof}

%%%%%%%%%%%%%%%%%%%%%%%%%%%%%%%%%%%%%%%%%%%%%%%%%%%

\section{An explicit characterization of the population barycenter} \label{sec:main}

In this section, we extend the results of Theorem \ref{theo:1D} to dimension $d \geq 2$. Let $\mu_{\btheta} \in \MM_{+}(\Omega)$ denote the  parametric random  measure with distribution $\P_{g}$, as defined in equation \eqref{eq:Pg}. Let $\mu_{0}$ be a fixed (reference) measure in $\MM_{+}(\Omega)$ admitting a density with respect to the Lebesgue measure on $\RR^{d}$. Then, by Proposition \ref{prop:Brenier}, there exists, for any $\theta \in \Theta$, a unique optimal mapping $T_{\theta} : \Omega \to \Omega$ such that
$
\mu_{\theta} = T_{\theta} \# \mu_{0},
$
where $ T_{\theta} = \nabla \phi_{\theta}$ $\mu_{0}$-almost everywhere, and  $\phi_{\theta} : \Omega \to \RR$ is a lower semi-continuous  convex (l.s.c) function.  Let $\Omega_{0}$ (resp.\ $\Omega_{\theta}$) be the support of $\mu_{0}$ (resp.\ $\mu_{\theta}$).

In  Theorem \ref{prop:deform:measure} below, we  give sufficient conditions on the expectation of  $T_{\btheta}$ which imply that the  barycenter of $\mu_{\btheta}$ is given by $\mu^{\ast} = \EE\left( T_{\btheta} \right) \# \mu_{0}$. This result  means that computing a barycenter in the Wasserstein space amounts to take the expectation (in the usual sense) of the optimal mapping  $T_{\btheta}$ to transport  $\mu_{0}$ onto $\mu_{\btheta}$. %The proof of this result relies on the use of a dual formulation of the optimization problem \eqref{eq:pbgen} that is stated in equation \eqref{eq:Past}, and we refer to Section \ref{sec:dual} for further details.

To state the main result of this section, we first need to introduce the mapping $\bar{T} : \Omega_{0} \to \Omega$ defined for $x \in \Omega_{0}$ by 
$$
\overline{T}(x) = \EE\left(T_{\btheta}(x)\right) = \int_{\Theta} T_{\theta}(x) g(\theta) d \theta.
$$
A key point in what follows is to assume that $\overline{T}$ is a $C^{1}$ diffeomorphism from the interior of $\Omega_{0}$ to the interior of its range $\overline{\Omega} = \overline{T}(\Omega_{0})$. To simplify the presentation, it is always understood that diffeomorphisms are defined on open sets although it will not always be mentioned. 
Then,  for any $\theta \in \Theta$, we introduce the mapping $\bar{T}_{\theta}$ defined for any $x$ in $\overline{\Omega}$ by
\begin{equation}
\bar{T}_{\theta}(x) = T_{\theta} \circ \bar{T}^{-1}(x). \label{eq:barvarphi}
\end{equation}
The next theorem is the main result of the paper.

%\begin{center}
%{\color{red} 
%{\bf Que peut-on dire sur $\bar{T}$ ? Est-ce une application optimale ? } \vspace{0.2cm} \\
%
%{\bf Existence de l'inverse de $\bar{T}$ ?} 
%}
%\end{center}

\begin{theo} \label{prop:deform:measure}
Let $\btheta \in \RR^{p}$ be a random vector with a density $g : \Theta \to \RR$  such that $g(\theta) > 0$ for all $\theta \in \Theta$. Let $\mu_{\btheta} = \Phi(\btheta)$ be a parametric random measure  with distribution $\P_{g}$. Let $\mu_{0}$ be a fixed measure in $\MM_{+}(\Omega)$ admitting a density   with respect to the Lebesgue measure on $\RR^{d}$. Suppose that %Assumption \ref{ass:Compact} and
Assumption \ref{ass:Phi}  hold.
 If for any   $\theta \in \Theta$  the following asumptions hold
\begin{eqnarray}
(i) & &  T_{\theta}  \mbox{ is  a $C^1$ diffeomorphism from $\Omega_{0}$ to $\Omega_{\theta}$}, \label{cond:diffeo1} \\
(ii) & &   \bar{T}   \mbox{ is  $C^1$ diffeomorphism from $\Omega_{0}$ to $\overline{\Omega}$ } , \label{cond:diffeo2} \\
(iii) & &  \bar{T}_{\theta}(x)   = \nabla \bar{\phi}_{\theta}(x) \mbox{ for all $x \in \overline{\Omega}$, } \label{eq:condadm}
\end{eqnarray}
where $\bar{\phi}_{\theta} : \Omega \to \RR$ is a l.s.c.\ convex function, that is such that, for any $x\in\Omega$, the function  $\theta \mapsto \bar{\phi}_{\theta}(x) $ is integrable with respect to $\P_{\Theta}$ and satisfies the normalization condition
\begin{equation}
\int_{\Theta}  \bar{\phi}_{\theta}(x) g(\theta) d \theta = \frac{1}{2} |x|^2 \mbox{ for all } x \in \Omega. \label{eq:condadm2}
\end{equation}
Then  the population barycenter is the measure $\mu^{\ast} \in \MM_{+}(\Omega)$   given by 
\begin{equation}
\mu^{\ast} = \overline{T}  \# \mu_{0}, \label{eq:baryTbar}
\end{equation}
and  the optimization problem   \eqref{eq:pbgen} satisfies
\begin{equation}
\inf_{\nu \in \MM_{+}(\Omega)}  J(\nu) =     \frac{1}{2} \int_{\Theta} d^{2}_{W_{2}}(\mu^{\ast}, \mu_{\theta}) g(\theta) d \theta  =  \frac{1}{2}  \int_{\Omega}  \EE \left(  | \bar{T}_{\btheta}(x)   - x |^{2} \right) d \mu^{\ast} (x). \label{eq:optiprimal}
\end{equation}
\end{theo}

\begin{proof}
  
Under the assumptions of Theorem \ref{prop:deform:measure}, it follows from Proposition \ref{prop:exist} that the barycenter $\mu^{\ast}$ exists and is unique. 
%
%Let us first prove that Condition \eqref{cond:diffeo} implies that $\bar{T}_{\theta} = T_{\theta} \circ \bar{T}^{-1}$ is a $C^1$ diffeomorphism. By  our remarks in Section \ref{subsec:T}, one has that $(x,\theta) \mapsto T_{\theta}(x)$ is a mesurable mapping with respect to the completion of $  \BB(\RR^{p})  \otimes \BB \left( \MM_{+}(\Omega) \right)$ with respect to $g(\theta) d\theta d \mu_{0}(x)$. Therefore, $x \mapsto \bar{T}(x)$ is also a mesurable mapping. Using Assumption \ref{ass:Compact} (compactness of $\Omega$), Condition \eqref{cond:diffeo}, and Lebesgue's dominated convergence theorem, it can be checked that $\bar{T}$  is a $C^1$ diffeomorphism of  $\Omega$, which implies that $\bar{T}_{\theta} = T_{\theta} \circ \bar{T}^{-1}$ is also a $C^1$ diffeomorphism  of $\Omega$  for any $\theta \in \Theta$. \\

The proof relies on the dual formulation \eqref{eq:Past} of the optimization problem \eqref{eq:pbgen}, and we refer to Section \ref{sec:dual} for further details and notation. To prove the results stated in Theorem \ref{prop:deform:measure}, we will use the dual characterization of the barycenter $\mu^{\ast}$ that is stated in Proposition \ref{prop:minmax}. For this purpose, we need to find a maximizer $f^{\Theta} = (f_{\theta})_{\theta \in \Theta} \in  L^{1}(\Theta, X)$ of the dual problem $(\PP^{\ast})$, see equation \eqref{eq:Past}.  In the proof, we  repeatedly use the fact that $\mu_{\theta} =  \bar{T}_{\theta} \#  \bar{\mu}$ where, by definition, $\bar{\mu} =  \bar{T} \# \mu_{0}$, and the property that  $\bar{T}_{\theta} = T_{\theta} \circ \bar{T}^{-1}$ is a $C^1$ diffeomorphism from $\overline{\Omega}$ to $\Omega_{\theta}$ which follows from Assumptions \eqref{cond:diffeo1} and \eqref{cond:diffeo2}. \\
 
\noindent a) Let us first compute an upper bound of $J_{\PP^{\ast}}$. Let $f^{\Theta} \in L^{1}(\Theta, X)$ be such that $\int_{\Theta} f_{\theta}(x)  d \theta = 0$ for all $x \in \Omega$. By definition of $S_{g(\theta)} f_{\theta}(x)$ one has that
\begin{equation} \label{eq:Sgtheta}
S_{g(\theta)} f_{\theta}(x) \leq  \frac{g(\theta)}{2} |x-y|^{2} - f_{\theta}(y)
\end{equation}
for any $y \in \Omega$. By using inequality \eqref{eq:Sgtheta} with $y =  \bar{T}_{\theta}^{-1}(x)$ for $x \in \Omega_{\theta}$, one obtains that
\begin{eqnarray*}
\int_{\Theta} \int_{\Omega} S_{g(\theta)} f_{\theta}(x) d \mu_{\theta}(x) d \theta & \leq & \int_{\Theta} \int_{\Omega_{\theta}} \left( \frac{g(\theta)}{2} |x- \bar{T}_{\theta}^{-1}(x)|^{2} -  f_{\theta}\left( \bar{T}_{\theta}^{-1}(x)   \right) \right)  d \mu_{\theta}(x)  d \theta \\
& = &  \int_{\Theta} \int_{\overline{\Omega}} \left( \frac{g(\theta)}{2} | \bar{T}_{\theta}(u)-u |^{2} -  f_{\theta} \left( u \right) \right)   d \bar{\mu}(u) d \theta \\
& = &  \int_{\Theta} \int_{\overline{\Omega}} \left( \frac{g(\theta)}{2} |\bar{T}_{\theta}(u) -u |^{2} \right) d \bar{\mu}(u) d \theta
\end{eqnarray*}
Note that to obtain the second inequality above, we have used the change of variable $u =\bar{T}_{\theta}^{-1}(x) $, while the third inequality has been obtained using  the fact that $ \int_{\Theta} f_{\theta} \left( u \right) d \theta = 0$ for any $u \in \Omega$ combined with Fubini's theorem. Thanks to Condition \eqref{eq:condadm}, $\bar{T}_{\theta}$ is the gradient of a convex function, and it is such that  $\mu_{\theta} =  \bar{T}_{\theta} \#  \bar{\mu}$. Hence, by Proposition \ref{prop:Brenier}, one obtains that $\int_{\Omega}   |\bar{T}_{\theta}(u) -u |^{2}   d \bar{\mu}(u)  = d^{2}_{W_{2}}(\bar{\mu}, \mu_{\theta}) $ which implies that
\begin{eqnarray*}
J(\bar{\mu}) & = &   \int_{\Theta}  \frac{1}{2} d^{2}_{W_{2}}(\bar{\mu}, \mu_{\theta}) g(\theta) d \theta,  \\
& = &  \int_{\Theta} \int_{\Omega} \left( \frac{g(\theta)}{2} |\bar{T}_{\theta}(u) -u |^{2} \right) d \bar{\mu}(u) d \theta =   \frac{1}{2}  \int_{\Omega}  \EE \left(  |\bar{T}_{\btheta}(u)- u |^{2} \right) d \bar{\mu}(u),
\end{eqnarray*}
where the last equality above is a consequence of Fubini's Theorem which follows from the measurability of the mapping $(u,\theta) \mapsto \bar{T}_{\theta}(u)$ that has been stated in Section \ref{subsec:T}.

Therefore, we have shown that
$$
\int_{\Theta} \int_{\Omega} S_{g(\theta)} f_{\theta}(x) d \mu_{\theta}(x) d \theta  \leq J(\bar{\mu}),
$$
for any $f^{\Theta} \in L^{1}(\Theta, X)$  such that $\int_{\Theta} f_{\theta}(x)  d \theta = 0$ for all $x \in \Omega$. This inequality  implies that
\begin{equation} \label{eq:boundJPast}
J_{\PP^{\ast}} \leq J(\bar{\mu}) = \frac{1}{2} \int_{\Omega}  \EE \left(  |\bar{T}_{\btheta}(u)- u |^{2} \right) d \bar{\mu}(u),
\end{equation}
where $J_{\PP^{\ast}}$ denotes the maximal value of the dual problem \eqref{eq:Past}.\\

\noindent b) Let us recall that we have assumed that $g(\theta) > 0$ for any $\theta \in \Theta$. Now, for any $\theta \in \Theta$, we define the function
\begin{equation}
f_{\theta}(x) = -g(\theta)   \left( \bar{\phi}_{\theta}(x)  - \frac{1}{2} |x|^2 \right), \; x \in \Omega, \label{eq:solutiondual}
\end{equation}
where $\bar{\phi}_{\theta}$ is the l.s.c.\ convex function introduced in Condition \eqref{eq:condadm2}.
First,  Assumption \eqref{eq:condadm} ensures that $f^{\Theta} = \left( f_{\theta} \right)_{\theta \in \Theta}$ belongs to $L^{1}(\Theta, X) $. Moreover, by  Condition \eqref{eq:condadm2}, one has  that $ \int_{\Theta} f_{\theta}(x)  d \theta  = 0 $ for all $x \in \Omega$.

Now,  for a given $\theta \in \Theta$, we  need to compute the value of $S_{g(\theta)} f_{\theta}(x)$ for any $x \in \Omega_{\theta}$. To this end, we introduce, for any $x \in \Omega_{\theta}$,  the function $F  : \Omega \to \RR$ defined by
$$
F(y)   = \frac{g(\theta)}{2} |x-y|^{2} - f_{\theta}(y)   = \frac{g(\theta)}{2} |x-y|^{2} +g(\theta)   \left( \bar{\phi}_{\theta}(y)  - \frac{1}{2} |y|^2 \right), \; y \in \Omega.
$$
Note that $F(y)   = \frac{g(\theta)}{2} |x|^{2} - g(\theta) \langle x , y \rangle  + g(\theta) \bar{\phi}_{\theta}(y)$. Hence,  $F$ is  a convex function, since, by assumption, $\bar{\phi}_{\theta}$ is convex.
Searching for some $y \in \overline{\Omega}$, where the gradient of $F$ vanishes, leads to the equation
\begin{eqnarray*}
0 & = & -g(\theta) (x-y) + g(\theta) \left( \nabla \bar{\phi}_{\theta}(y)  -  y  \right) = -g(\theta) (x-y) +g(\theta)  \left(  \bar{T}_{\theta}(y)  -  y  \right),
\end{eqnarray*} %
Since $g(\theta) > 0$, it follows that the convex function $y \mapsto F(y)$ has a minimum at $y =\bar{T}_{\theta}^{-1}(x) \in \overline{\Omega}$. Therefore,
\begin{eqnarray}
S_{g(\theta)} f_{\theta}(x) & = &   \frac{g(\theta)}{2} |x-\bar{T}_{\theta}^{-1}(x) |^{2} +g(\theta)   \left( \bar{\phi}_{\theta}(\bar{T}_{\theta}^{-1}(x) )  - \frac{1}{2} |\bar{T}_{\theta}^{-1}(x)|^2 \right),   \label{eq:Sthetabar}
\end{eqnarray}
for any $x \in \Omega_{\theta}$.

Let us introduce the notation $J^{\ast}\left( f^{\Theta} \right) = \int_{\Theta} \int_{\Omega} S_{g(\theta)} f_{\theta}(x) d \mu_{\theta}(x)   d \theta  $. By equation \eqref{eq:Sthetabar} and using the change of variable $u =\bar{T}_{\theta}^{-1}(x) $, one obtains that
\begin{eqnarray*}
J^{\ast}\left( f^{\Theta} \right)  & =  & \int_{\Theta} \int_{\Omega} \left(  \frac{g(\theta)}{2} |x-\bar{T}_{\theta}^{-1}(x) |^{2} +g(\theta)   \left( \bar{\phi}_{\theta}(\bar{T}_{\theta}^{-1}(x) )  - \frac{1}{2} |\bar{T}_{\theta}^{-1}(x)|^2 \right) \right) d \mu_{\theta}(x)   d \theta  \\
& = &  \int_{\Theta} \int_{\Omega} \left(  \frac{g(\theta)}{2} |\bar{T}_{\theta}(u)-u |^{2} +g(\theta)   \left( \bar{\phi}_{\theta}(u )  - \frac{1}{2} |u|^2 \right) \right)d \bar{\mu}(u)   d \theta \\
& = &  \int_{\Theta} \int_{\Omega}   \frac{g(\theta)}{2} |\bar{T}_{\theta}(u)-u |^{2}  d \bar{\mu}(u) =   \frac{1}{2}   \int_{\Omega}  \EE \left(  |\bar{T}_{\btheta}(u)- u |^{2} \right) d \bar{\mu}(u) d \theta,
\end{eqnarray*}
where the last equality follows as an application of Fubini's theorem combined with Condition \eqref{eq:condadm2}. Hence, thanks to the upper bound \eqref{eq:boundJPast}, we finally have that
$$
J^{\ast}\left( f^{\Theta} \right) = J_{\PP^{\ast}} = \frac{1}{2} \int_{\Omega}  \EE \left(  |\bar{T}_{\btheta}(u)- u |^{2} \right) d \bar{\mu}(u) = J(\bar{\mu}),
$$
which proves that $f^{\Theta}$ is a maximizer of the dual problem $(\PP^{\ast})$. Moreover, by Proposition \ref{prop:minmax}, one also obtains that that $J(\bar{\mu}) \leq J(\nu)$ for any $\nu \in \MM_{+}(\Omega)$. Hence, $\bar{\mu} = \mu^{\ast}$ is the barycenter of the random measure $\mu_{\btheta}$. This completes  the proof of Theorem \ref{prop:deform:measure}. 
\end{proof}

Theorem \ref{prop:deform:measure} shows that, under appropriate assumptions,  computing the population barycenter in the Wasserstein space of the parametric random measure  $\mu_{\btheta}$  amounts to transport the reference measure  $\mu_{0}$ by the expected amount of deformation measured by $\overline{T}$. We discuss below Assumptions  \eqref{cond:diffeo1}, \eqref{cond:diffeo2} and \eqref{eq:condadm}  stated in Theorem \ref{prop:deform:measure}.

Assumption \eqref{cond:diffeo1}  holds under appropriate smoothness conditions both on the measure $\mu_{\theta}$ and $\mu_{0}$ and their the supports $\Omega_{0}$ and $\Omega_{\theta}$. Such assumptions are difficult to summarize in a general setting. For a detailed review on this issue, we refer to the discussion in \cite{Figalli}. Nevertheless, in the next subsection, we  provide various examples of statistical models for which such an assumption holds.

The next proposition give sufficient conditions  ensuring that Assumption \eqref{cond:diffeo2} holds, namely  the existence of $\bar{T}$ as a $C^{1}$ diffeomorphism.

\begin{prop}\label{prop:hypi} 
Assume that for any $\theta\in\Theta,\ T_{\theta}$ is  $C^1$ diffeomorphism from $\Omega_{0}$ to $\Omega_{\theta}$. Let $\overline{T}(x) = \EE\left(T_{\btheta}(x)\right) = \int_{\Theta} T_{\theta}(x) g(\theta) d \theta , \; x \in \Omega_{0}$. Suppose that $\Omega_{0}$ is a convex set, and assume that for any $x\in\Omega_{0}$ there exists $\epsilon>0$ and an integrable function $K : \Theta \to \RR$ (with respect to $g(\theta) d\theta$) such that for any $y\in B(x,\epsilon)$ one has that
\begin{equation}
\left|\frac{\partial}{\partial x}T_\theta(y)\right|\leq  K(\theta). \label{eq:K}
\end{equation}
Then, $\overline{T}$ is a $C^1$ diffeormorphism from $\Omega_{0}$ to $\overline{\Omega}$.
\end{prop}
\begin{proof}
   By  our remarks in Section \ref{subsec:T}, one has that $(x,\theta) \mapsto T_{\theta}(x)$ is a mesurable mapping with respect to the completion of $  \BB(\RR^{p})  \otimes \BB \left( \MM_{+}(\Omega) \right)$ with respect to $g(\theta) d\theta d \mu_{0}(x)$. Therefore, $x \mapsto \bar{T}(x)$ is also a mesurable mapping. Now, by Assumption \eqref{eq:K}, one can apply the Theorem of differentiability under the integral to obtain that $\overline{T}$ is $C^1$. Moreover, we have that 
   $$
   \nabla \overline{T}(x)=\int_{\Theta} \nabla \overline T_{\theta}(x) g(\theta) d \theta,
   $$
   for any $x \in \Omega_{0}$.    Since $ T_{\theta}$ is an optimal mapping and a  $C^1$ diffeomorphism, it follows that  $ T_{\theta}$ is the gradient of a  strictly convex function, and therefore we get that $\nabla \overline{T}(x)$ is  a positive  definite matrix. Hence, by the local inversion Theorem, we know that, for any $x \in \Omega_{0}$,  $\overline{T}$ is a $C^1$ diffeomorphism in a neighborhood of $x$. It remains now to prove that $x\mapsto\overline{T}(x)$ is injective. Take $x_1$ and $x_2$ two distinct points in $\Omega_{0}$. Since, by assumption, $\Omega_{0}$ is a convex set, the segment $[x_1,x_2]$ is included in $\Omega_{0}$. Now, for each $\theta \in \Theta$, the restriction of $T_\theta$ to the segment $[x_1,x_2]$ remains the derivative of a strictly convex function in dimension one.  Thus, $t \mapsto \overline{T}(t x_1 + (1-t) x_{2}) =  \int_{\Theta} T_{\theta}(t x_1 + (1-t) x_{2}) g(\theta) d \theta$ is a strictly increasing function on $[0,1]$. Hence, it follows $T(x_1) \neq T(x_2)$ and thus   $x\mapsto\overline{T}(x)$ is injective.
 \end{proof}

In the case where $\overline{T}(x) = x$ for all $x \in \Omega_{0}$,  it can be checked that  Assumption  \eqref{eq:condadm} is necessarily satisfied. This corresponds to the situation where the population  barycenter $\mu^{\ast}$ is equal to the reference measure $\mu_{0}$.  Note that when $\overline{T} = I$, it is not required in Theorem  \ref{prop:deform:measure} that the set of optimal maps $(T_{\theta})_{\theta \in \Theta}$ are admissible in the sense of Definition 4.2  in  \cite{MR3338645}. Indeed, according to this definition, a necessary condition for a countable collection of one-to-one maps  $(T_{\theta})_{\theta \in \Theta}$ to be an admissible class is that $T_{\theta} \circ T_{\theta'}^{-1}$ is an optimal mapping (in the sense of Proposition \ref{prop:Brenier}) for any $\theta, \theta' \in  \Theta$. In the case where $\overline{T} = I$,  such an assumption is thus not necessary.

When $\overline{T} \neq I$, Assumption \eqref{eq:condadm}  is not necessarily satisfied since the composition $T_{\theta} \circ \bar{T}^{-1}$ between the   optimal mapping $T_{\theta}$ and the map $\bar{T}^{-1}$    is not always an optimal mapping. In the following section, we describe statistical models to illustrate the usefulness of  the barycenter in the Wasserstein space for data analysis, and we discuss various assumptions on $\overline{T}$ to ensure that Assumption \eqref{eq:condadm} is satisfied. Condition  \eqref{eq:condadm2} is rather an identifiability condition. Indeed, by definition of $\overline{T}$, one always has that $\EE ( \bar{T}_{\btheta}(x) ) = \int_{\Theta}  \bar{T}_{\theta}(x) g(\theta) d \theta   = x$ for all $x \in \Omega$. Hence, among the various convex functions $ \bar{\phi}_{\theta}$ such that $\bar{T}_{\theta}   = \nabla \bar{\phi}_{\theta}$, Condition  \eqref{eq:condadm2} implies to choose the one that ``integrates'' the relation $\int_{\Theta}  \bar{T}_{\theta}(x) g(\theta) d \theta   = x$ without any additional constant term.

Finally, it should be also remarked that Condition \eqref{eq:condadm} is weaker than requiring that $T_{\theta} \circ T_{\theta'}^{-1}$ is an optimal mapping for any $\theta, \theta' \in  \Theta$. Therefore,  even in the case where $\overline{T} \neq I$, the conditions stated in Theorem  \ref{prop:deform:measure}  to characterize a population barycenter by averaging optimal transport maps are weaker than those given in  \cite{MR3338645}, since we do not require that the random optimal maps $(T_{\theta})_{\theta \in \Theta}$  belong to an admissible class of mappings (in the sense of Definition 4.2 in \cite{MR3338645}). Moreover, we recall that the study in \cite{MR3338645} is  restricted to the characterization of empirical barycenters and their asymptotic properties. 

%\begin{center}
%{\color{red} 
%{\bf Commenter également Condition \eqref{cond:diffeo} ! } 
%}
%\end{center}

\section{An application to deformable models in statistics} \label{sec:appli}

In this section, we propose to discuss some applications of Theorem \ref{prop:deform:measure}. To define probability models where averaging in the Wasserstein space amounts to take the expectation of an optimal transport map,  we study  statistical models for which the notion of population and empirical barycenters in the 2-Wasserstein space is relevant.

\subsection{General framework} \label{sec:models}

In many applications observations are in the form of a set of $n$ one-dimensional signals or gray-level  images $\bX_{1},\ldots,\bX_{n}$
(e.g.\ in geophysics, biomedical imaging or in signal processing for neurosciences), which can be 
considered as iid random variables belonging to an appropriate space $\FF(\Omega)$ of real-valued functions on a compact domain of $\Omega \subset \RR^{d}$. In many situations the observed curves or images share the same structure. This may lead to the assumption that these observations are random elements which vary around a  reference template. Characterizing and estimating such a template is then of fundamental interest in many applications.

In the presence of geometric variability in time or space in the data, a widely used approach is Grenander's pattern theory \cite{Gre,gremil,MR2176922,TYShape} that models such a variability by the action of a Lie group on an infinite dimensional space of curves or images.   Following the ideas of Grenander's pattern theory, a simple assumption is to consider that the  data $\bX_{1},\ldots,\bX_{n}$ are obtained through the deformation of the same template $h \in \FF(\Omega)$  via the so-called deformable model
\begin{equation}
\bX_{i} = h \circ \bvarphi^{-1}_{i}, \; i=1,\ldots,n, \label{eq:deformmodel}
\end{equation}
where $\bvarphi_{1},\ldots,\bvarphi_{n}$ are iid random variables belonging to the set of smooth diffeomorphisms of $\Omega$.  In signal and image processing, there has been recently a growing interest on the statistical analysis of deformable models (in the presence of additive noise) using either rigid or non-rigid random diffeomorphisms $\bvarphi_{i}$, see e.g.\  \cite{MR2301497, BC11,MR2676894, BGL09,BLV09,MR2369028, MR2662362} and the review proposed in \cite{ABGMR13}. Nevertheless, in a data set of curves or images, one generally observes not only a source of variability in geometry, but also a source of photometric variability (e.g.\ the intensity of a pixel changes from one image to another) that cannot be only captured by a deformation of the domain $\Omega$ via a diffeomorphism as in model \eqref{eq:deformmodel}.
 
It is always possible to transform the data $\bX_{1},\ldots,\bX_{n}$ into a set of $n$ iid random probability densities by computing the random variables
$$
\bq_{i}(x) = \frac{\widetilde{\bX_{i}}(x)}{\int_{\Omega} \widetilde{\bX_{i}}(u)du} , \; x \in \Omega, \mbox{ where } \widetilde{\bX_{i}}(x) = \bX_{i}(x) - \min_{u \in \Omega} \left\{ \bX_{i}(u) \right\}, \; i=1,\ldots,n.
$$
Let $q_{0} \in \FF(\Omega)$ be a probability density function. In this section and as already discussed in the introduction, we consider the following deformable model of densities:
\begin{equation}
\bq_{i}(x)  = \left| \det\left( D \bvarphi^{-1}_{i}  \right) (x) \right|  q_{0}\left(  \bvarphi^{-1}_{i}(x) \right) \; x \in \Omega, \; i=1,\ldots,n, \label{eq:modelwarping}
\end{equation}
where $\det\left( D \bvarphi^{-1}_{i}  \right) (x)$ denotes the determinant of the Jacobian matrix of the random diffeomorphism $ \bvarphi^{-1}_{i} $ at point $x$. If we denote by $\bmu_{1},\ldots,\bmu_{n} \in \MM_{+}(\Omega)$ the random probability measures  with densities $\bq_{1},\ldots,\bq_{n}$, and by $\mu_{0}$ the measure with density $q_{0}$, then \eqref{eq:modelwarping} can also be written as the following deformable model of measures
\begin{equation}
\bmu_{i} =  \bvarphi_{i}  \# \mu_{0}, \; i=1,\ldots,n. \label{eq:modelwarping:measures}
\end{equation}
In model \eqref{eq:modelwarping:measures}, under appropriate assumptions on the random maps $\bvarphi_{1},\ldots,\bvarphi_{n}$, computing the empirical barycenter in the Wasserstein space of the random measures $\bmu_{1},\ldots,\bmu_{n}$ leads to a consistent and meaningful estimator of the reference measure $\mu_{0}$ and thus of the template $q_{0}$. In the rest of this section, we discuss some examples of model \eqref{eq:modelwarping:measures}. The main contribution is to  show how  Theorem \ref{prop:deform:measure} can be used to characterise the population barycenter of random measures satisfying the deformable model \eqref{eq:modelwarping:measures}.

\subsection{A parametric class of diffeomorphisms}

Let $\mu_{0}$ be a measure on $\RR^{d}$ having a density $q_{0}$ (with respect to the Lebesgue measure $dx$ on $\RR^{d}$) whose support is equal to a compact set $\Omega_{0} \subset \RR^{d}$.
We propose to characterise the population barycenter of a random measure $\bmu$ satisfying the deformable model
$
\bmu  =  \bvarphi  \# \mu_{0}, % \label{eq:deform:measure}
$
for a specific class of random diffeomorphisms $\bvarphi : \RR^{d} \to \RR^{d}$. Let $\Ss_{d}^{+}(\RR)$ be the set of positive definite $d \times d$ symmetric matrices  with real entries.  Let 
\begin{equation}
\phi : (\RR^{p}, \BB(\RR^{p})) \to \left(\Ss_{d}^{+}(\RR) \times \RR^{d}, \BB \left( \Ss_{d}^{+}(\RR) \times \RR^{d}  \right) \right) \label{eq:mapphi}
\end{equation}
be a measurable mapping, where $\BB( \Ss_{d}^{+}(\RR) \times \RR^{d} )$ is the Borel $\sigma$-algebra of $\Ss_{d}^{+}(\RR) \times \RR^{d}$. For $\theta \in \RR^{p}$, we will use the notations
$$
\phi(\theta) = \left(A_{\theta}, b_{\theta} \right), \mbox{ with } A_{\theta} \in \Ss_{d}^{+}(\RR), \; b_{\theta} \in \RR^{d},
$$
and
$$
T_{\theta}(x) = A_{\theta} x + b_{\theta}, \; x \in \RR^{d}.
$$
Since the matrix $A_{\theta}$ belongs to $\Ss_{d}^{+}(\RR)$, it is clear that $T_{\theta}$ is the gradient of the convex function
$$
\phi_{\theta}(x) = \frac{1}{2} x' A_{\theta} x + b_{\theta}'x, \mbox{ for } x \in \RR^{d}.
$$
Moreover, for any $\theta \in \RR^{d}$, it follows that $T_{\theta}  : \RR^{d} \to \RR^{d}$ is a $C^1$ diffeomorphism with inverse given by
$$
T_{\theta}^{-1}(x) =   A_{\theta}^{-1} \left(x - b_{\theta} \right), \; x \in \RR^{d}.
$$
Let $\Theta \subset \RR^{p}$ be a compact set, 
%One can then define a  parametric class  of diffeomorphisms of $\RR^{d}$  as follows
%\begin{equation}
%D_{\phi}(\Theta) = \{ T_{\theta}, \; \theta \in \Theta\}. \label{eq:Dphi}
%\end{equation}
and let $\btheta \in \RR^{p}$ be a random vector with density $g$ (with respect to the Lebesgue measure $d \theta$ on $\RR^{p}$) having a support included in $\Theta$. We propose to study the population barycenter in the Wasserstein space of the random measure $\mu_{\btheta}$ satisfying the  deformable model:
\begin{equation}
\mu_{\btheta}  =  T_{\btheta}  \# \mu_{0}.  \label{eq:deform:measure:theta}
\end{equation}
The above equation may also be interpreted as a semi-parametric model of random densities, where $T_{\btheta}$ is the optimal mapping between $\mu_{0}$ and $\mu_{\btheta}$. For any $\theta \in \Theta$ (not necessarily a random vector), we define $\mu_{\theta}  =  T_{\theta}  \# \mu_{0}$. Since $T_{\theta}$ is a smooth diffeomorphism and $\mu_{0}$ is a measure with density $q_{0}$ whose support is  the compact set $\Omega_{0}$, it follows that $\mu_{\theta}$ admits a density $q_{\theta}$ on $\RR^{d}$ given by
\begin{equation}
q_{\theta}(x)  =
\left\{\begin{array}{ccc}
 \det \left( A_{\theta}^{-1}  \right) q_{0} \left(  A_{\theta}^{-1} \left(x - b_{\theta} \right) \right) & \mbox{ if } & x \in \Omega_{\theta}, \\
0 & \mbox{ if } & x \notin \Omega_{\theta}.
\end{array}
\right. \label{eq:qtheta}
\end{equation}
where $  \Omega_{\theta} = \left\{  T_{\theta} ( y ), y \in \Omega_{0} \right\} = \left\{  A_{\theta} y + b_{\theta}, y \in \Omega_{0} \right\}$. It is clear that $T_{\theta}$ is a $C^{1}$ diffeormorphism form $\Omega_{0}$ to $\Omega_{\theta}$ that is the optimal mapping between $ \mu_{0}$ and $\mu_{\theta}$.

Moreover, from the compactness of $\Omega_{0}$, it follows that $\Omega_{\theta}$ is compact  for all $\theta \in \Theta$. If we further assume that the mapping $\phi : \Theta \to \Ss_{d}^{+}(\RR) \times \RR^{d}$ is continuous, one obtains (by the compactness assumption on $\Theta$) that there exists a radius $r > 0$ such that   $  \Omega_{\theta} \subset \Omega = \overline{B(0,r)}$ for all $\theta \in \Theta$. Thus, under this assumption, the random measure $\mu_{\btheta}$ takes its values in $\MM_{+}(\Omega)$.  

As a consequence of  Theorem \ref{prop:deform:measure}, one has the following result.
\begin{coro}
Assume that the mapping $\phi$ defined in \eqref{eq:mapphi} is continuous, and let $\Omega =   \overline{B(0,r)}$ be  such that $  \Omega_{\theta} \subset \Omega$ for all $\theta \in \Theta$. If
$$
\EE(A_{\btheta}) = I \mbox{ and } \EE(b_{\btheta}) = 0
$$
where $I$ is the identity matrix, then $\mu_{0} = \mu^{\ast}$ is the population barycenter in the Wasserstein space of the random measure $\mu_{\btheta} \in \MM_{+}(\Omega)$ satisfying the deformable model \eqref{eq:deform:measure:theta}.
\end{coro}

In the case where the condition $\EE(A_{\btheta})  = I$ is not necessarily satisfied, we define, for any $\theta \in \Theta$, the following quantities
\begin{equation}
\bar{A}_{\theta} = A_{\theta} \bar{A}^{-1} \quad \mbox{ and }  \quad \bar{b}_{\theta} = b_{\theta} - A_{\theta} \bar{A}^{-1} \bar{b}, \label{eq:barAb}
\end{equation}
where $\bar{A}= \EE \left( A_{\btheta} \right)$ and $\bar{b} = \EE \left( b_{\btheta} \right)$. For any $\theta \in \Theta$, we also define the mapping
$$
\bar{T}_{\theta}(x) = T_{\theta} (\bar{T}^{-1}(x)) = \bar{A}_{\theta} x + \bar{b}_{\theta}, \; x \in \RR^{d},
$$
where $\bar{T}(x) = \bar{A}x + \bar{b} $. To apply Theorem \ref{prop:deform:measure}, it is necessary to assume that $\bar{T}_{\theta}$ is the gradient of a convex function, which means assuming that $\bar{A}_{\theta} = A_{\theta} \bar{A}^{-1}$ is a positive definite matrix. A sufficient condition to have this property is to assume that the product $A_{\theta} \bar{A}^{-1}$ is symmetric for any $\theta \in \Theta$, since the product of two matrices in $\Ss_{d}^{+}(\RR)$ is positive definite if and only if their product is symmetric (see e.g.\ \cite{Meenakshi19993}). Under such assumptions, it is then possible to apply Theorem \ref{prop:deform:measure} to obtain  following result.

\begin{coro}
Assume that the mapping $\phi$ defined in \eqref{eq:mapphi} is continuous, and let $\Omega =   \overline{B(0,r)}$ be  such that $  \Omega_{\theta} \subset \Omega$ for all $\theta \in \Theta$. If
$
A_{\theta} \bar{A}^{-1} \mbox{ is symmetric for any } \theta \in \Theta,
$
then
$$
\mu^{\ast} = \bar{T} \# \mu_{0}.
$$
is the population barycenter in the Wasserstein space of the random measure $\mu_{\btheta} \in \MM_{+}(\Omega)$ satisfying the deformable model \eqref{eq:deform:measure:theta}.
\end{coro}
 
\section{Convergence of the empirical barycenter} \label{sec:empbary}

Let us now prove the convergence of the empirical barycenter for the set of  measures considered in this paper whose supports are included in the compact set $\Omega$.  We recall that this assumption implies that the Wasserstein space $(\MM_{+}(\Omega),d_{W_{2}})$ is compact. This is the key property that we use to study the convergence of the empirical barycenter. 

Let $\btheta_{1},\ldots,\btheta_{n}$ be iid random variables in $\RR^{p}$ with distribution $\P_{\Theta}$. Then, let us define the functional
\begin{eqnarray}
J_{n}(\nu) & = &   \frac{1}{n} \sum_{j=1}^{n} \frac{1}{2} d^{2}_{W_{2}}(\nu, \mu_{\btheta_{j}} ),   \; \nu \in \MM_{+}(\Omega) \label{eq:Jn}, 
\end{eqnarray}
and consider the optimization problem: find an empirical barycenter
\begin{equation}
\bar{\bmu}_{n} \in \argmin_{\nu \in \MM_{+}(\Omega)} J_{n}(\nu),    \label{eq:pbemp}
\end{equation}
Thanks to the results in \cite{MR2801182}, the following lemma holds:
\begin{lem} \label{lem:exists} Suppose that Assumption \ref{ass:Phi} holds.
Then, for any $n \geq 1$, there exists a unique minimizer $\bar{\bmu}_{n}$ of  $J_{n}(\cdot)$ over $\MM_{+}(\Omega)$.
\end{lem}

Let us now give our main result on the convergence of the empirical barycenter $\bar{\bmu}_{n}$.

\begin{theo} \label{theo:conv}
Suppose that   %Assumption \ref{ass:Compact} and
Assumption \ref{ass:Phi}   hold. Let $\mu^{\ast}$ be the population barycenter defined by \eqref{eq:pbgen}, and $\bar{\bmu}_{n}$ be  the empirical barycenter defined by \eqref{eq:pbemp}. Then,
$$
\lim_{n \to + \infty} d_{W_{2}}(\bar{\bmu}_{n},\mu^{\ast}) = 0 \mbox{ almost surely (a.s.)}
$$
\end{theo}

\begin{proof}
For $\nu \in \MM_{+}(\Omega)$, let us define
$$
\Delta_{n}(\nu) = J_{n}(\nu) - J(\nu).
$$
The proof is divided in two steps. It follows standard arguments to obtain a strong law of large number for random variables belonging to compact metric spaces that can be found in  \cite{MR600540}. Theorem \ref{theo:conv} can be seen as  a particular case of the general results in  \cite{MR600540}. But in order to be self contained  we include a detailed proof that is inspired by the arguments in \cite{MR600540}.  First, we prove the uniform convergence to zero of $\Delta_{n}$ over $\MM_{+}(\Omega)$. Then, we show that any converging subsequence of  $\bar{\bmu}_{n}$ converges a.s.\ to  $\mu^{\ast}$ for the 2-Wasserstein distance. \\

\noindent Step 1.  For $ \nu \in \MM_{+}(\Omega)$, let us denote by $f_{\nu} :  \MM_{+}(\Omega) \to \RR$ the real-valued function  defined by
$$
f_{\nu}(\mu) = \frac{1}{2} d^{2}_{W_{2}}(\nu,\mu). 
$$
Then, let us define the following class of functions
$$
\FF = \left\{ f_{\nu}, \; \nu \in \MM_{+}(\Omega)  \right\}.
$$

Let $\delta^2(\Omega) = \sup_{x \in \Omega} \{ |x|^2\}$. Since $\Omega$ is compact, it follows that $\FF$ is a  class of functions uniformly bounded by $2\delta^{2}(\Omega)$ (for the supremum norm). Now, let $\nu, \mu, \mu' \in \MM_{+}(\Omega)$. By the triangle reverse inequality
\begin{eqnarray*}
|f_{\nu}(\mu) - f_{\nu}(\mu')| & = & \frac{1}{2} \left| d^{2}_{W_{2}}(\nu,\mu) - d^{2}_{W_{2}}(\nu,\mu')  \right| \leq  \sqrt{2} \delta(\Omega)\left| d_{W_{2}}(\nu,\mu) - d_{W_{2}}(\nu,\mu')  \right| \\
& \leq &    \sqrt{2} \delta(\Omega) d_{W_{2}}(\mu,\mu').
\end{eqnarray*}
The above inequality proves that $\FF$ is an equicontinuous family of functions. Now, let $\btheta_{1},\ldots,\btheta_{n}$ be iid random vectors in $\RR^{p}$ with density $g$, and let us define the random empirical measure on $\left(\MM_{+}(\Omega), \BB \left( \MM_{+}(\Omega) \right) \right)$
$$
\P_{g}^{n} = \frac{1}{n} \sum_{i=1}^{n} \delta_{\mu_{\btheta_{i}}} =  \frac{1}{n} \sum_{i=1}^{n} \delta_{\phi\left(\btheta_{i}\right)} ,
$$
where $\delta_{\mu}$ denotes the Dirac measure. It is clear that
$$
\Delta_{n}(\nu) = \int_{\MM_{+}(\Omega)}  f_{\nu}(\mu) d \P_{g}^{n}(\mu) -  \int_{\MM_{+}(\Omega)}  f_{\nu}(\mu) d \P_{g} (\mu).
$$
Let  $f : \MM_{+}(\Omega) \to \RR$ be a real-valued function that is continuous (for the topology induced by $d_{W_{2}}$) and bounded. Thanks to the mesurability of the mapping $\phi$, one has that the real random variable $  \int_{\MM_{+}(\Omega)}  f(\mu) d \P_{g}^{n}(\mu)$ converges a.s.\ to $  \int_{\MM_{+}(\Omega)}  f(\mu) d \P_{g}(\mu)$ as $n \to + \infty$, meaning that the random measure $\P_{g}^{n}$ a.s.\ converges to $\P_{g}$ in the weak sense. Therefore, since $\FF$ is a uniformly bounded and equicontinuous family of functions, one can use Theorem 6.2 in \cite{MR0137809} to obtain that
\begin{equation}
\sup_{\nu \in \MM_{+}(\Omega)} \left| \Delta_{n}(\nu) \right| = \sup_{f \in \FF}  \left| \int_{\MM_{+}(\Omega)}  f(\mu) d \P_{g}^{n}(\mu) -  \int_{\MM_{+}(\Omega)}  f(\mu) d \P_{g} (\mu) \right| \to 0 \mbox{ as } n \to + \infty, \; a.s. \label{eq:convunif}
\end{equation}
which proves the uniform convergence of $\Delta_{n}$ to zero over $\MM_{+}(\Omega)$. \\

\noindent Step 2. By Lemma \ref{lem:exists}, there exists a unique sequence $(\bar{\bmu}_{n})_{n \geq 1}$ of empirical barycenters defined by \eqref{eq:pbemp}. Thanks to the compactness of the Wasserstein space $(\MM_{+}(\Omega),d_{W_{2}})$, one can extract a converging sub-sequence of empirical barycenters  $(\bar{\bmu}_{n_{k}})_{k \geq 1}$ such that $\lim_{k \to + \infty} d_{W_{2}}(\bar{\bmu}_{n_{k}}, \bar{\bmu}) = 0$ for some measure $\bar{\bmu} \in \MM_{+}(\Omega)$.

Let us now prove that $\bar{\bmu} = \mu^{\ast}$. To this end, let us first note that by the definition of $\bar{\bmu}_{n_{k}}$ and $\mu^{\ast}$ as the unique minimizer of $J_{n_{k}}(\cdot)$ and $J(\cdot)$ respectively, it follows that
\begin{eqnarray*}
\left| J(\bar{\bmu}_{n_{k}}) -  J(\mu^{\ast}) \right|  & =  & J(\bar{\bmu}_{n_{k}})  - J_{n_{k}}(\bar{\bmu}_{n_{k}}) + J_{n_{k}}(\bar{\bmu}_{n_{k}}) - J_{n_{k}}(\mu^{\ast}) + J_{n_{k}}(\mu^{\ast}) -  J(\mu^{\ast}) \\
& \leq & 2 \sup_{\nu \in \MM_{+}(\Omega)} \left| \Delta_{n_{k}}(\nu) \right|,
\end{eqnarray*}
where we have used the fact that $J_{n_{k}}(\bar{\bmu}_{n_{k}}) - J_{n_{k}}(\mu^{\ast}) \leq 0$. Therefore, thanks to  the uniform convergence  \eqref{eq:convunif}   of $\Delta_{n}$ to zero over $\MM_{+}(\Omega)$, one obtains that
\begin{equation}
\lim_{k \to + \infty} J(\bar{\bmu}_{n_{k}})  = J(\mu^{\ast}). \label{eq:conv}
\end{equation}
Therefore, using that
\begin{eqnarray*}
\left|  J_{n_{k}}(\bar{\bmu}_{n_{k}})  - J(\mu^{\ast}) \right| & \leq  & \left|  J_{n_{k}}(\bar{\bmu}_{n_{k}})  - J(\bar{\bmu}_{n_{k}})\right|  + \left|  J(\bar{\bmu}_{n_{k}})- J(\mu^{\ast}) \right|  \\
& \leq & \sup_{\nu \in \MM_{+}(\Omega)} \left| \Delta_{n_{k}}(\nu) \right| + \left|  J(\bar{\bmu}_{n_{k}})- J(\mu^{\ast}) \right|,
\end{eqnarray*}
one finally obtains by  \eqref{eq:convunif}    and \eqref{eq:conv} that
\begin{equation}
\lim_{k \to + \infty} J_{n_{k}}(\bar{\bmu}_{n_{k}})   = J(\mu^{\ast}) \label{eq:conv0}.
\end{equation}
Since $|J_{n_{k}}(\bar{\bmu}) - J(\bar{\bmu})| \leq \sup_{\nu \in \MM_{+}(\Omega)} \left| \Delta_{n_{k}}(\nu) \right| $,    it follows by equation \eqref{eq:convunif}  that
\begin{equation}
\lim_{k \to + \infty} J_{n_{k}}(\bar{\bmu}) = J(\bar{\bmu}) \; a.s. \label{eq:conv1}
\end{equation}
Moreover, for any $\epsilon > 0$, there exists $k_{\epsilon} \in \N$ such that $d_{W_{2}}(\bar{\bmu}_{n_{k}}, \bar{\bmu}) \leq \epsilon$ for all $k \geq k_{\epsilon}$. Therefore, using the triangle inequality, it follows that for all $k \geq k_{\epsilon}$
\begin{eqnarray*}
\left( J_{n_{k}}(\bar{\bmu})  \right)^{1/2} & = & \left( \frac{1}{n_{k}} \sum_{j=1}^{n_{k}} \frac{1}{2} d^{2}_{W_{2}}(\bar{\bmu}, \mu_{\btheta_{j}} )  \right) ^{1/2} \\
& \leq &  \left( \frac{1}{n_{k}} \sum_{j=1}^{n_{k}} \frac{1}{2} d^{2}_{W_{2}}(\bar{\bmu}_{n_{k}}, \mu_{\btheta_{j}} )  \right) ^{1/2} +  \left( \frac{1}{n_{k}} \sum_{j=1}^{n_{k}} \frac{1}{2} d^{2}_{W_{2}}(\bar{\bmu}, \bar{\bmu}_{n_{k}} )  \right) ^{1/2} \\
& \leq &  \left( \frac{1}{n_{k}} \sum_{j=1}^{n_{k}} \frac{1}{2} d^{2}_{W_{2}}(\bar{\bmu}_{n_{k}}, \mu_{\btheta_{j}} )  \right) ^{1/2}  + \frac{\epsilon}{\sqrt{2}},
\end{eqnarray*}
and thus by equations \eqref{eq:conv0} and \eqref{eq:conv1}, we obtain that
 \begin{equation}
J(\bar{\bmu})    \leq  \lim_{k \to + \infty}  J_{n_{k}}(\bar{\bmu}_{n_{k}}) =  J(\mu^{\ast})   \; a.s. \label{eq:conv2}
\end{equation}
%By definition of $\bar{\bmu}_{n_{k}}$ one has that $J_{n_{k}}(\bar{\bmu}_{n_{k}}) \leq J_{n_{k}}(\nu)$ for any $\nu  \in \MM_{+}(\Omega)$. Thus, by combining \eqref{eq:convunif} and \eqref{eq:conv2}, it follows that  for any $\nu  \in \MM_{+}(\Omega)$
% \begin{equation}
%J(\bar{\bmu})    \leq  J(\nu)   \; a.s. \label{eq:conv4},
%\end{equation}
which finally proves that $\bar{\bmu} = \mu^{\ast}$ a.s.\ since $\mu^{\ast}$ is the unique minimizer of $J(\nu)$ over $\nu  \in \MM_{+}(\Omega)$.

Hence, any converging subsequence of empirical barycenters  converges a.s.\ to  $\mu^{\ast}$ for the 2-Wasserstein distance. Since $(\MM_{+}(\Omega),d_{W_{2}})$ is compact, this finally shows that $(\bar{\bmu}_{n})_{n \geq 1}$ is a converging sequence such that $\lim_{n \to + \infty} d_{W_{2}}(\bar{\bmu}_{n}, \mu^{\ast}) = 0$ a.s.\  which completes the proof of Theorem \ref{theo:conv}.

\end{proof}

\section{Beyond the compactly supported case} \label{sec:persp}

To conclude the paper, we briefly discuss the case of a random measure $\mu_{\btheta}$ with  distribution $\P_{\Theta}$ whose support is not included in a compact set $\Omega$ of $\RR^{d}$.  In the one-dimensional case i.e.\ $d=1$, let  us denote by  $F_{\mu_{\btheta}}$ its cumulative distribution function, and by $F_{\mu_{\btheta}}^{-1}$ its generalized inverse (quantile function). Provided that one can define the  measure  $\mu^{\ast}  $  with quantile function  $F_{\mu^{\ast}}^{-1}(y) = \EE  \left( F^{-1}_{\mu_{\btheta}}(y)  \right)$ for all $y \in [0,1]$, it is expected that  arguments similar to those used in the proof of Theorem \ref{theo:1D} can be used to prove that $\mu^{\ast}$ is the unique population barycenter of the random measure $\mu_{\btheta}$  with distribution $\P_{\Theta}$.  

The multi-dimensional case (i.e.\ $d \geq 2$) is more involved. Indeed, the arguments that we used to prove the main results of the paper for $d \geq 2$ strongly depend on the compactness assumption on the support of the random measure $\mu_{\btheta}$. Adapting these arguments for the extension of this paper to non-compactly supported measures is an interesting topic for future investigations.

\bibliographystyle{plain}
\bibliography{barycenter_optimal_transport_revESAIM_PS}

\begin{thebibliography}{10}

\bibitem{Afsari}
B.~Afsari.
\newblock Riemannian $l^{p}$ center of mass: existence, uniqueness, and
  convexity.
\newblock {\em Proceedings of the American Mathematical Society},
  139(2):655--673, 2011.

\bibitem{MR2801182}
M.~Agueh and G.~Carlier.
\newblock Barycenters in the {W}asserstein space.
\newblock {\em SIAM J. Math. Anal.}, 43(2):904--924, 2011.

\bibitem{ACLL15}
A.~Agull\'o-Antol\'in, J.A. Cuesta-Albertos, H.~Lescornel, and J.-M. Loubes.
\newblock A parametric registration model for warped distributions with
  {W}asserstein distance.
\newblock {\em Journal of Multivariate Analysis}, 135:117--130, 2015.

\bibitem{MR2301497}
S.~Allassonni{\`e}re, Y.~Amit, and A.~Trouv{\'e}.
\newblock Towards a coherent statistical framework for dense deformable
  template estimation.
\newblock {\em J. R. Stat. Soc. Ser. B Stat. Methodol.}, 69(1):3--29, 2007.

\bibitem{ABGMR13}
S.~Allassonnière, J.~Bigot, J.~Glaunès, F.~Maire, and F.~Richard.
\newblock Statistical models for deformable templates in image and shape
  analysis.
\newblock {\em Annales Mathématiques Blaise Pascal,}, 20(1):1--35, 2013.

\bibitem{MR2814414}
P.~C. {\'A}lvarez-Esteban, E.~del Barrio, J.~A. Cuesta-Albertos, and
  C.~Matr{\'a}n.
\newblock Uniqueness and approximate computation of optimal incomplete
  transportation plans.
\newblock {\em Ann. Inst. Henri Poincar\'e Probab. Stat.}, 47(2):358--375,
  2011.

\bibitem{fixed-point}
P.~C. {\'A}lvarez-Esteban, E.~del Barrio, J.~A. Cuesta-Albertos, and
  C.~Matr{\'a}n.
\newblock A fixed-point approach to barycenters in {W}asserstein space.
\newblock {\em Journal of Mathematical Analysis and Applications},
  441(2):744--762, 2016.

\bibitem{MR2914758}
M.~Arnaudon, C.~Dombry, A.~Phan, and L.~Yang.
\newblock Stochastic algorithms for computing means of probability measures.
\newblock {\em Stochastic Process. Appl.}, 122(4):1437--1455, 2012.

\bibitem{2014-Benamou}
J.{-}D. Benamou, G.~Carlier, M.~Cuturi, L.~Nenna, and G.~Peyr{\'{e}}.
\newblock Iterative bregman projections for regularized transportation
  problems.
\newblock {\em {SIAM} J. Scientific Computing}, 37(2), 2015.

\bibitem{batach1}
R.~Bhattacharya and V.~Patrangenaru.
\newblock Large sample theory of intrinsic and extrinsic sample means on
  manifolds (i).
\newblock {\em Annals of statistics}, 31(1):1--29, 2003.

\bibitem{batach2}
R.~Bhattacharya and V.~Patrangenaru.
\newblock Large sample theory of intrinsic and extrinsic sample means on
  manifolds (ii).
\newblock {\em Annals of statistics}, 33:1225--1259, 2005.

\bibitem{BC11}
J.~Bigot and B.~Charlier.
\newblock On the consistency of {F}réchet means in deformable models for curve
  and image analysis.
\newblock {\em Electronic Journal of Statistics}, 5:1054--1089, 2011.

\bibitem{MR2676894}
J.~Bigot and S.~Gadat.
\newblock A deconvolution approach to estimation of a common shape in a shifted
  curves model.
\newblock {\em Ann. Statist.}, 38(4):2422--2464, 2010.

\bibitem{BGL09}
J.~Bigot, S.~Gadat, and J.M. Loubes.
\newblock Statistical $\mbox{M}$-estimation and consistency in large deformable
  models for image warping.
\newblock {\em Journal of Mathematical Imaging and Vision}, 34(3):270--290,
  2009.

\bibitem{BLV09}
J.~Bigot, J.M. Loubes, and M.~Vimond.
\newblock Semiparametric estimation of shifts on compact {L}ie groups for image
  registration.
\newblock {\em Probability Theory and Related Fields}, 152:425--473, 2010.

\bibitem{W1}
S.~Bobkov and M.~Ledoux.
\newblock {\em One-dimensional empirical measures, order statistics and
  Kantorovich transport distances}.
\newblock Memoirs of the American Mathematical Society, 2017.
\newblock Available at
  {https://perso.math.univ-toulouse.fr/ledoux/files/2016/12/MEMO.pdf}.

\bibitem{MR3338645}
E.~Boissard, T.~Le~Gouic, and J.-M. Loubes.
\newblock Distribution's template estimate with {W}asserstein metrics.
\newblock {\em Bernoulli}, 21(2):740--759, 2015.

\bibitem{2013-Bonneel-barycenter}
N.~Bonneel, J.~Rabin, G.~Peyr{\'e}, and H.~Pfister.
\newblock Sliced and radon wasserstein barycenters of measures.
\newblock {\em Journal of Mathematical Imaging and Vision}, 51(1):22--45, 2015.

\bibitem{Brenier91}
Y.~Brenier.
\newblock Polar factorization and monotone rearrangement of vector-valued
  functions.
\newblock {\em Comm. Pure Appl. Math.}, 44(4):375--417, 1991.

\bibitem{cuturi-14}
M.~Cuturi and A.~Doucet.
\newblock Fast computation of wasserstein barycenters.
\newblock In {\em In Tony Jebara and Eric P. Xing, editors, Proceedings of the
  31st International Conference on Machine Learning (ICML-14)}, pages 685--693.
  JMLR Workshop and Conference Proceedings, 2014.

\bibitem{Figalli}
G.~De~Philippis and A.~Figalli.
\newblock The {M}onge-{A}mp\`ere equation and its link to optimal
  transportation.
\newblock {\em Bull. Amer. Math. Soc.}, 51:527--580, 2014.

\bibitem{MR1727362}
I.~Ekeland and R.~T{\'e}mam.
\newblock {\em Convex analysis and variational problems}, volume~28 of {\em
  Classics in Applied Mathematics}.
\newblock Society for Industrial and Applied Mathematics (SIAM), Philadelphia,
  PA, english edition, 1999.
\newblock Translated from the French.

\bibitem{MR2643592}
J.~Fontbona, H.~Gu{\'e}rin, and S.~M{\'e}l{\'e}ard.
\newblock Measurability of optimal transportation and strong coupling of
  martingale measures.
\newblock {\em Electron. Commun. Probab.}, 15:124--133, 2010.

\bibitem{fre}
M.~Fr{\'e}chet.
\newblock Les {\'e}l{\'e}ments al{\'e}atoires de nature quelconque dans un
  espace distanci{\'e}.
\newblock {\em Ann. Inst. H.Poincar{\'e}, Sect. B, Prob. et Stat.},
  10:235--310, 1948.

\bibitem{MR2369028}
F.~Gamboa, J.-M. Loubes, and E.~Maza.
\newblock Semi-parametric estimation of shifts.
\newblock {\em Electron. J. Stat.}, 1:616--640, 2007.

\bibitem{MR1108330}
C.~Goodall.
\newblock Procrustes methods in the statistical analysis of shape.
\newblock {\em J. Roy. Statist. Soc. Ser. B}, 53(2):285--339, 1991.

\bibitem{Gre}
U.~Grenander.
\newblock {\em General pattern theory - {A} mathematical study of regular
  structures}.
\newblock Clarendon Press, Oxford, 1993.

\bibitem{gremil}
U.~Grenander and M.~Miller.
\newblock {\em Pattern Theory: From Representation to Inference}.
\newblock Oxford Univ. Press, Oxford, 2007.

\bibitem{MR2058553}
S.~Haker and A.~Tannenbaum.
\newblock On the {M}onge-{K}antorovich problem and image warping.
\newblock In {\em Mathematical methods in computer vision}, volume 133 of {\em
  IMA Vol. Math. Appl.}, pages 65--85. Springer, New York, 2003.

\bibitem{Haker04optimalmass}
S.~Haker, L.~Zhu, A.~Tannenbaum, and S.~Angenent.
\newblock Optimal mass transport for registration and warping.
\newblock {\em International Journal on Computer Vision}, 60:225--240, 2004.

\bibitem{MR2816349}
S.~F. Huckemann.
\newblock Intrinsic inference on the mean geodesic of planar shapes and tree
  discrimination by leaf growth.
\newblock {\em Ann. Statist.}, 39(2):1098--1124, 2011.

\bibitem{kendall}
D.G. Kendall.
\newblock Shape manifolds, procrustean metrics, and complex projective spaces.
\newblock {\em Bull. London Math Soc.}, 16:81--121, 1984.

\bibitem{KimPass}
Y.~H. Kim and B.~Pass.
\newblock Wasserstein barycenters over {R}iemannian manifolds.
\newblock {\em Advances in Mathematics}, 307:640--683, 2017.

\bibitem{legouic:hal-01163262}
T.~Le~Gouic and J.-M. Loubes.
\newblock {Existence and Consistency of Wasserstein Barycenters}.
\newblock {\em {Probability Theory and Related Fields}}, 168(3-4):901 -- 917,
  2017.

\bibitem{MR0288601}
A.~D. Loffe and V.~M. Tihomirov.
\newblock Duality of convex functions and extremum problems.
\newblock {\em Uspehi Mat. Nauk}, 23(6 (144)):51--116, 1968.

\bibitem{Meenakshi19993}
A.R. Meenakshi and C.~Rajian.
\newblock On a product of positive semidefinite matrices.
\newblock {\em Linear Algebra and its Applications}, 295(1-3):3 -- 6, 1999.

\bibitem{MR3004954}
B.~Pass.
\newblock Optimal transportation with infinitely many marginals.
\newblock {\em J. Funct. Anal.}, 264(4):947--963, 2013.

\bibitem{rabin-ssvm-11}
J.~Rabin, G.~Peyr{\'e}, J.~Delon, and M.~Bernot.
\newblock Wassertein barycenter and its applications to texture mixing.
\newblock In {\em LNCS, Proc. SSVM'11}, volume 6667, pages 435--446. Springer,
  2011.

\bibitem{MR0137809}
R.~Ranga~Rao.
\newblock Relations between weak and uniform convergence of measures with
  applications.
\newblock {\em Ann. Math. Statist.}, 33:659--680, 1962.

\bibitem{MR2039961}
K.-T. Sturm.
\newblock Probability measures on metric spaces of nonpositive curvature.
\newblock In {\em Heat kernels and analysis on manifolds, graphs, and metric
  spaces ({P}aris, 2002)}, volume 338 of {\em Contemp. Math.}, pages 357--390.
  Amer. Math. Soc., Providence, RI, 2003.

\bibitem{MR600540}
H.~Sverdrup-Thygeson.
\newblock Strong law of large numbers for measures of central tendency and
  dispersion of random variables in compact metric spaces.
\newblock {\em Ann. Statist.}, 9(1):141--145, 1981.

\bibitem{MR2176922}
A.~Trouv{\'e} and L.~Younes.
\newblock Local geometry of deformable templates.
\newblock {\em SIAM J. Math. Anal.}, 37(1):17--59 (electronic), 2005.

\bibitem{TYShape}
A.~Trouv{\'e} and L.~Younes.
\newblock Shape spaces.
\newblock In {\em Handbook of Mathematical Methods in Imaging}. Springer, 2011.

\bibitem{villani2003tot}
C.~Villani.
\newblock {\em {Topics in Optimal Transportation}}.
\newblock American Mathematical Society, 2003.

\bibitem{MR2662362}
M.~Vimond.
\newblock Efficient estimation for a subclass of shape invariant models.
\newblock {\em Ann. Statist.}, 38(3):1885--1912, 2010.

\end{thebibliography}

\end{document}